\documentclass[12pt]{amsart}
\usepackage{amsfonts}
\usepackage{amssymb}
\usepackage{color}
\usepackage{dsfont}
\usepackage{setspace}
\usepackage[fleqn,tbtags]{mathtools}
\numberwithin{equation}{section}
\theoremstyle{plain}
 \newtheorem{theorem}{Theorem}[section]
 \newtheorem{lemma}[theorem]{Lemma}
 \newtheorem{corollary}[theorem]{Corollary}
 \newtheorem{proposition}[theorem]{Proposition}

\theoremstyle{definition}
 \newtheorem{definition}[theorem]{Definition}
 \newtheorem{example}[theorem]{Example}
 \newtheorem{remark}[theorem]{Remark}

\newcommand{\R}{\mathbb{R}}

\newcommand{\bN}{\mathbb{N}}
\newcommand{\bZ}{\mathbb{Z}}
\newcommand{\bR}{\mathbb{R}}

\newcommand{\bC}{\mathbb{C}}
\newcommand{\cB}{\mathcal{B}}

\newcommand{\cF}{\mathcal{F}}

\newcommand{\CLK}{\mathrm{CLK}_0^n}
\newcommand{\CLKeins}{\mathrm{CLK}_0^1}
\newcommand{\GFSC}{\mathrm{GFS}(\bR^n,\bC)}
\newcommand{\GFSA}{\mathrm{GFS}(\bR^n,A)}
\newcommand{\GFSQ}{\mathrm{GFS}(\bR^{n-q},A)}
\newcommand{\CAP}{\mathrm{CAP}(\bR^n,A)}

\newcommand{\re}{\mathrm{e}}
\newcommand{\ri}{\mathrm{i}}
\newcommand{\di}{\mathrm{d}}
\newcommand{\one}{\mathbf{1}}

\allowdisplaybreaks

\begin{document}
\author{David Berger}\thanks{David Berger (corresponding author): Technische Universit\"at Dresden,
Institut f\"ur Mathematische Stochastik, D-01062 Dresden, Germany, david.berger2@tu-dresden.de}
\author{Alexander Lindner}\thanks{Alexander Lindner: Ulm University, Institute of Mathematical Finance, D-89081 Ulm, Germany,
alexander.lindner@uni-ulm.de}
\title{Invertible complex  measures on Euclidian space}
\maketitle

\begin{abstract}
In 1971 Taylor characterised all complex measures on $\bR$ that are invertible with respect to convolution as those which can be written in the form $\delta_\gamma \ast \sigma^{\ast m} \ast \exp(\nu)$ for some $\gamma\in \bR$, some complex measure $\nu$, some $m\in \bZ$ and a given fixed invertible finite signed measure $\sigma$ (which has characteristic function $\bR \ni z \mapsto (1+\ri z)/(1-\ri z)$). We extend Taylor's result to complex measures on $\bR^n$. Somewhat surprisingly, the structure of invertible complex measures on $\bR^n$ is not much more complicated than that of complex measures on $\bR$, in the sense that they can be represented as $\delta_\gamma \ast \sigma_1^{\ast m_1} \ast \ldots \ast \sigma_p^{\ast m_p} \ast \exp(\nu)$ for some $\gamma \in \bR^n$, some complex measure $\nu$ and $m_1,\ldots, m_p\in \bZ$, where the $\sigma_i$ correspond to $\sigma$ in the one-dimensional case and actually live on $1$-dimensional subspaces of $\bR^n$. Our proof relies on a general result of Taylor for invertible complex measures on locally compact abelian groups. To apply Taylor's result, we extend some existing results for $\bC$-valued functions to functions with values in a semisimple commutative unital Banach algebra with connected Gelfand space. The study of invertible complex measures on $\bR^n$ has some impact on the theory of quasi-infinitely divisible probability distributions on $\bR^n$.
\end{abstract}

\section{Introduction} \label{S-Introduction}
In the following, by a \emph{measure} we will always mean a positive measure, i.e.~a
$\sigma$-additive $[0,\infty]$-valued set-function on a $\sigma$-algebra (assigning the value 0 to the empty set), while by \emph{complex} or \emph{signed measures} we mean $\bC$-valued or $[-\infty,\infty]$-valued $\sigma$-additive set-functions (assigning the value 0 to the empty set). In particular, a complex measure is always finite. Denote by $\cB^n$ the usual Borel $\sigma$-algebra on $\bR^n$ and
by $M(\bR^n)$ the set of all complex measures on $(\bR^n,\cB^n)$.  Equipped with the total variation norm and convolution as multiplication, $M(\bR^n)$ becomes a commutative unital Banach algebra, with identity $\delta_0^n$, the Dirac measure at $0$ in $\bR^n$.
We are interested in the question of when $\mu$ is \emph{invertible}, i.e.~when there exists a complex measure $\mu' \in M(\bR^n)$ such that $\mu' \ast \mu = \delta_0^n$, or in other words, when $0$ is outside the spectrum of $\mu$. It is well known that an easy description of the spectrum  of $M(\bR^n)$ and the corresponding Gelfand space is very difficult (cf.~Taylor~\cite[Rem.~1.3.5]{T73}). Denote by $\widehat{\mu}$ the characteristic function, i.e.~the Fourier transform of $\mu$, defined by
\begin{equation} \label{eq-def-FT}
\widehat{\mu}:\bR^n \to \bC, \quad z \mapsto \widehat{\mu}(z) := \int_{\bR^n} \re^{\ri z^T x} \, \mu(\di x),
\end{equation}
where $z^T$ denotes the transpose of the column vector $z\in \bR^n$.
Then each $\mu\in M(\bR^n)$ is uniquely determined by $\widehat{\mu}$, and $\widehat{\mu}$ is clearly continuous and bounded on $\bR^n$. Hence, if $\mu$ is invertible, say $\mu \ast \mu' = \delta_0^n$ with $\mu'  \in M(\bR^n)$, then $\widehat{\mu'}$ is bounded, hence $\widehat{\mu}$ must be bounded away from zero, i.e.
\begin{equation} \label{eq-bounded-cf}
\inf_{z\in \bR^d} |\widehat{\mu}(z)| > 0
\end{equation}
is necessary for $\mu$ to be invertible. In particular, by the Riemann--Lebesgue lemma, an absolutely continuous (with respect Lebesgue measure) complex  measure  can never be invertible. One might hope that condition \eqref{eq-bounded-cf} is also sufficient for $\mu$ to be invertible, but this is not the case, see Taylor~\cite[Prop. 1.3.3]{T73} for a counterexample. On the other hand, an easy sufficient condition for $\mu$ to be invertible is that $\mu$ is the exponential of
some complex measure $\nu\in M(\bR^n)$, i.e.~that
\begin{equation} \label{eq-exponential}
\mu = \exp(\nu) := \sum_{j=0}^\infty \frac{1}{j!} \nu^{\ast j} , \quad \nu \in M(\bR^n),
\end{equation}
where $\nu^{\ast j}$ denotes the $j$-fold convolution of $\nu$ with itself (with $\nu^{\ast 0} := \delta_0^n$). For if $\mu = \exp (\nu)$, then also
$\exp(-\nu) \in M(\bR^n)$ and
$$\exp(-\nu) \exp(\nu) = \exp(-\nu+\nu) = \exp (0_{M(\bR^n)}) = \delta_0^n,$$
showing that $\mu$ is invertible (where $0_{M(\bR^n)}$ denotes the zero-measure on $\bR^n$). The question of what are necessary and sufficient conditions for $\mu$ to be invertible is much more involved. For complex measures on $\bR$, a complete characterisation was obtained in 1971 by Taylor~\cite[Cor. 4.3]{T71}, see also his expositions in~\cite{T72,T73}.

\begin{theorem} \label{t-Taylor1} {\rm (Taylor~\cite[Cor. 4.3]{T71}, also \cite[Cor. 4.8]{T72} and \cite[Thm. 8.4.1]{T73}).} \\
Denote by $\sigma$ the finite signed measure on $(\bR^1,\cB^1)$ whose characteristic function is given by $\widehat{\sigma}(z) = (1+\ri z)/(1-\ri z)$, $z\in \bR$, i.e.~$\sigma(\di s) :=   2 \re^{-s} \one_{(0,\infty)}(s) \, \di s - \delta_0^1 (\di s)$. Then a complex measure $\mu\in M(\bR)$ is invertible if and only if there are $m\in \bZ$, $\gamma \in \bR$ and $\nu\in M(\bR)$ such that
$$\mu = \delta_\gamma^1 \ast \sigma^{\ast m} \ast \exp(\nu).$$
\end{theorem}
Here, $\sigma^{\ast m}$ when $m \in \{-1,-2,\ldots\}$ is to interpreted as $(\sigma^{-1})^{\ast |m|}$, where $\sigma^{-1}$ is the inverse of $\sigma$ in $M(\bR)$, given by $\sigma^{-1} (\di s) = 2 \re^{s} \one_{(-\infty,0)}(s) \, \di s - \delta_0^1 (\di s)$, with characteristic function $\widehat{\sigma^{-1}}(z) = (1-\ri z)/(1+\ri z)$, $z\in \bR$. Further, the Dirac measure at a point $\gamma\in \bR^n$ will always be denoted by $\delta_\gamma^n$.

We see that invertible complex measures on $\bR$ are \lq almost\rq~given by $\exp(\nu)$, the only missing ingredients being the shift $\delta_\gamma^1$ and $\sigma^{\ast m}$. Taylor obtained his result as a consequence of a more general result on invertible regular complex measures on locally compact abelian groups. We state it here only for the case when $G$ is the locally compact abelian group $\bR^n$ (with addition and the usual euclidian topology). In the following, for a general topology $\tau$ on $\bR^n$ we write $\bR^n_\tau$ to denote the group $\bR^n$ (under addition) equipped with the topology $\tau$, while if the topology is the euclidian topology we usually simply write $\bR^n$. Further terms used in the following result will be explained in Section~\ref{S-prelim}.

\begin{theorem} \label{t-Taylor2} {\rm (Taylor~\cite[Cor. 4.2]{T71}, also~\cite[Thm. 3]{T72} and \cite[Thm.~8.2.4]{T73}).}\\
Let $\mu \in M(\mathbb{R}^n)$ be invertible. Then there exist  $p\in\mathbb{N}_0$, $\nu\in M(\bR^n)$, topologies $\tau_1,\dotso,\tau_p$ on $\bR^n$, all at least as strong as the euclidean topology and such that $\bR^n_{\tau_i}$ remains a locally compact abelian group for all $i\in \{1,\ldots, p\}$, and complex measures $\mu_1,\ldots, \mu_p \in M(\bR^n)$ such that
$\mu$ is of the form
\begin{align*}
\mu= \mu_1 \ast \dotso\ast  \mu_p \ast \exp(\nu),
\end{align*}
where each $\mu_i$ is invertible and of the form $\mu_i = \alpha_i \delta_0^n + \widetilde{\mu_i}_{|\cB^n}$, where $\alpha_i \in \bC$ and $\widetilde{\mu_i}$ is a regular complex  measure on $(\bR^n_{\tau_i}, \cB(\bR^n_{\tau_i}))$ that is absolutely continuous with respect to the Haar measure on $\bR^n_{\tau_i}$.
\end{theorem}

Actually, Taylor~\cite{T71,T72,T73,T-spectrum} also shows that the inverses of the $\mu_i$ are of the same form as the $\mu_i$. We will take this point up in more detail in Proposition~\ref{p-Taylor4}.

While Theorem~\ref{t-Taylor2} is the key ingredient for studying invertibility of complex measures on $\bR^n$,  it still leaves some questions open, in particular the determination of the invertible complex measures $\mu_i$ appearing there. In Section~\ref{S2} we will state our main results, giving a complete characterisation of invertible complex measures on~$\bR^n$, thus generalising Theorem~\ref{t-Taylor1} to higher dimensions. Somewhat surprisingly, it turns out that the structure of invertible complex measures on $\bR^n$ is not much more complicated than the structure of those on $\bR$.

Our original motivation for studying invertibility of  complex measures comes from the theory of quasi-infinitely divisible probability distributions. These are probability measures $\mu$ on $\bR^n$ for which there exist two infinitely divisible distributions $\mu_1$ and $\mu_2$ such that $\mu_1 \ast \mu = \mu_2$. They can be characterised as probability measures whose characteristic function has  a L\'evy--Khintchine type representation with a quasi-L\'evy measure, this is a \lq signed L\'evy measure\rq~rather than a (positive) L\'evy measure, see~\cite{BKL21,LPS2018}  for more information on such probability measures. It is already evident from the definition that there should be some relations between invertibility of complex measures and quasi-infinitely divisible distributions, and in fact using invertible complex measures has turned out a very helpful tool for identifying various classes of quasi-infinitely divisible distributions, see e.g.~Alexeev and Khartov~\cite[Thm.~3.2]{AlexeevKhartov23}, Berger~\cite[Proof of Thm.~4.12]{B19} or Berger and Kutlu~\cite[Thm.~2.1]{BK23}. Further literature on quasi-infinitely divisible distributions includes Alexeev and Khartov~\cite{AlexeevKhartov23b}, Cuppens~\cite{Cuppens1975}, Khartov~\cite{Khartov2024,Khartov2025}, Nakamura and Suzuki~\cite{NakamuraSuzuki}, Passeggeri~\cite{Passeggeri} or Zhang et al.~\cite{ZhangLiKerns}, to name just a few.
Besides its applications to quasi-infinitely divisible distributions, invertibility of complex measures is clearly related to invertibility of convolution operators, see Taylor~\cite[Prop.~1.2.5, Cor.~1.2.8]{T73}, and it helps to clarify the spectrum  of elements of $M(\bR^n)$. Finally, we find  the question of invertibility of complex measures on $\bR^n$ interesting in its own right. Recent studies related to invertibility of complex measures include
Ohrysko and Wasilewski~\cite{OW2020} and Raubenheimer and van Appel~\cite{Raubenheimer}.

The paper is organised as follows. In the next section, we will state our main results, which give characterisations of invertible complex measures on $\bR^n$, thus extending Theorem~\ref{t-Taylor1} to higher dimensions. The proof of these results relies on Theorem~\ref{t-Taylor2}, but to apply it, the invertible complex measures $\mu_i$  appearing there have to be determined, which in some cases (i.e. for some topologies) is easy or already done (namely for the \lq boundary cases\rq~$q=0$ and $q=n$), but becomes very involved for the case $q\in \{1,\ldots, n-1\}$ (the meaning of $q$ being explained in Section~\ref{S3}). For the treatment of the latter cases, we have to extend a result of Alexeev and Khartov~\cite[Thm.~3.2]{AlexeevKhartov23} on the invertibility of characteristic functions of discrete complex measures on $\bR^n$ to a setting where invertibility of  Banach algebra-valued generalized  Fourier series are considered (under certain conditions). For that purpose, we derive some  results for Banach algebra-valued functions which we believe to be interesting in their own right, such as the existence of a distinguished logarithm of such functions. Let us  specify the structure of the paper in more detail. After stating the main results in Section~\ref{S2},
we present some notation and preliminaries in Section~\ref{S-prelim}, while in Section~\ref{S3} we  give the general plan for the proof of the main results and determine the form of the complex measures $\mu_i$ appearing in Theorem~\ref{t-Taylor2}, however without characterising invertibility of them. Sections~\ref{S4}--\ref{S8} are concerned with invertibility of these complex measures. In Section~\ref{S4} we treat the case of absolutely continuous complex measures plus a point mass (corresponding to $q=n$), while in Section~\ref{S5} Alexeev and Khartov's result on discrete complex measures, in combination with a result of Taylor~\cite{T73}, is cited (corresponding to the case $q=0$). Section~\ref{S6} contains all results we develop for Banach algebra-valued functions, including the aforementioned generalisation of Alexeev and Khartov's result to Banach algebra-valued generalized Fourier series, while in Section~\ref{S7} we apply these results for the treatment of the intermediate case $q\in \{1,\ldots, n-1\}$. The final proof of the main results follows in Section~\ref{S8}. Finally, in Section~\ref{S9} we present an application of our results to divisibility properties of invertible complex measures.

\section{Statement of the main results} \label{S2}

The following is our first main result, which directly extends Theorem~\ref{t-Taylor1} to higher dimensions. In its statement and everywhere else in the paper, the product measure of two complex measures $\rho_1$ and $\rho_2$ (or $\sigma$-finite positive measures) will be denoted by $\rho_1 \otimes \rho_2$, and by $U(\rho)$ we denote the image measure of a  complex measure or positive measure $\rho \in M(\bR^n)$ under the mapping $\bR^n \to \bR^n$, $x\mapsto Ux$, induced by a matrix $U\in \bR^{n\times n}$.

\begin{theorem}[Characterisation of invertible complex measures via factorisation] \label{t-main1}
Let $n\in \bN$. Denote by $\sigma \in M(\bR)$ the finite signed measure on $\bR$ with characteristic function $\widehat{\sigma}(z) = (1+\ri z)/(1-\ri z)$, i.e.~$\sigma (\di s) = 2 \re^{-s} \one_{(0,\infty)}(s) \di s - \delta_0^1 (\di s)$. Let $n\in \bN$. Then a complex measure $\mu \in M(\bR^n)$ is invertible if and only if there are $\nu\in M(\bR^n)$, $\gamma \in \bR^n$,  $p\in \bN_0$, integers $m_1,\ldots, m_p$ and orthogonal matrices $U_1,\ldots, U_p \in O(n)$ such that
    \begin{equation} \label{eq-repr1}
    \mu = \delta_{\gamma}^n \ast U_1  (\sigma^{\ast m_1} \otimes \delta_0^{n-1}) \ast \ldots \ast
    U_p  (\sigma^{\ast m_p} \otimes \delta_0^{n-1}) \ast \exp(\nu).
    \end{equation}
\end{theorem}

It is remarkable that the structure of invertible complex measures in higher dimensions turns out to be not much more complicated than that of invertible complex measures on $\bR$. Indeed, the $1$-dimensional signed measure $\sigma$ also appeared in Theorem~\ref{t-Taylor1}, and $U_i(\sigma^{\ast m_i} \otimes \delta_0^{n-1})$ embeds the signed measure $\sigma^{\ast m_i}$ into the one-dimensional line $U_i(\bR \times \{0_{n-1}\}) \subset \bR^n$. Observe that when $n=1$, there are only two orthogonal matrices on $\bR$, namely $(1) \in \bR^{1\times 1}$ and $(-1) \in \bR^{1\times 1}$, and since the image measure of $\sigma^{\ast m}$ under the orthogonal matrix $(-1)$ is just $\sigma^{\ast (-m)}$, Theorem~\ref{t-main1} reduces to Theorem~\ref{t-Taylor1} when $n=1$.

When $\mu \in M(\bR^n)$ is not generally complex-valued but even real-valued (i.e.~a finite signed measure on $(\bR^n,\cB^n)$), one is interested in a representation of the invertible $\mu$ with a signed rather than complex measure $\nu$ in \eqref{eq-repr1}. Such a representation can be achieved, if additionally the factor $\mu(\bR^n)$ is allowed:

\begin{corollary}[Characterisation of invertible finite signed measures via factorisation] \label{c-main2}
Let $\sigma$ and $n\in \bN$ be as in Theorem~\ref{t-main1} and $\mu\in M(\bR^n)$ be a finite signed measure, different from the zero-measure. Then $\mu$ is invertible if and only there are a finite signed measure $\nu\in M(\bR^n)$, $\gamma \in \bR^n$,  $p\in \bN_0$, integers $m_1,\ldots, m_p$ and orthogonal matrices $U_1,\ldots, U_p \in O(n)$ such that
    \begin{equation*} \label{eq-repr2}
    \mu = \mu(\bR^n) \, \delta_{\gamma}^n \ast U_1  (\sigma^{\ast m_1} \otimes \delta_0^{n-1}) \ast \ldots \ast
    U_p  (\sigma^{\ast m_p} \otimes \delta_0^{n-1}) \ast \exp(\nu).
    \end{equation*}
The inverse  of $\mu$ is then a finite signed measure as well.
\end{corollary}

While Theorem~\ref{t-main1} gives a characterisation of invertible complex measures in terms of a factorisation, one is often interested in the characteristic function of invertible complex measures. In particular, when considering quasi-infinitely divisible probability distributions, where our original motivation stems from, then characteristic functions are the key ingredient for their study. For the corresponding characterisation, recall that a \emph{L\'evy measure} on $\bR^n$ is a  (positive) measure $\rho$ on $(\bR^n,\cB^n)$ such that $\rho(\{0\}) = 0$ and $\int_{\bR^n} \min(1,|x|^2) \, \rho(\di x) < \infty$ (see~\cite[Sect.~8]{Sato2013} or~\cite[Chap.~9]{BrockLind2024} for more information on L\'evy measures). We will mostly be interested in the L\'evy measures that do not only integrate $\min(1,|x|^2)$, but even $\min(1,|x|)$.
 We will however need complex L\'evy type measures, and in particular those which integrate $\min(1,|x|)$. They are defined as follows:

\begin{definition}(see  \cite[Def.~3.1]{B19}, \cite[Def.~2.1]{BKL21}, \cite[Def.~2.2]{LPS2018}) \label{d-quasi}
Let $\cB^n_0 := \{ B \in \cB^n : \exists r > 0: B \cap \{ x \in \bR^n : |x|\leq r\} = \emptyset\}$, the system of all Borel sets in $\bR^n$ that are bounded away from zero. Then a set-function $\nu:\cB_0^n \to \bC$ is called a \emph{complex L\'evy type measure} if there are four L\'evy measures $\rho_i:\cB^n \to [0,\infty]$, $i=1,\ldots, 4$, such that
$$\nu(B) = \rho_1(B) - \rho_2(B) + \ri (\rho_3(B) - \rho_4(B)) \quad \forall\; B \in \cB^n_0.$$
(Observe that $\rho_i(B) < \infty$ for all $B\in \cB^n_0$). If the $\rho_i$ can be chosen such that even $\int_{\bR^n} \min(1,|x|) \, \rho_i(\di x) < \infty$ for all $i=1,\ldots, 4$, we shall call  $\nu$ a \emph{complex L\'evy type measure integrating $\min(1,|x|)$}. For such $\nu$ (integrating $\min(1,|x|)$) and a given $\cB^n$-measurable function $f:\bR^n \to \bC$ satisfying $|f(x)| \leq C \min (1,|x|)$ for all $x\in \bR^n$ for some finite constant~$C$ we then define
\begin{align}
 & {\int_{\bR^n} f(x) \, \nu(\di x)} \label{eq-def-int}\\
 := & \int_{\bR^n} f(x) \, \rho_1(\di x) - \int_{\bR^n} f(x) \, \rho_2(\di x) + \ri \left( \int_{\bR^n} f(x) \, \rho_3(\di x) - \int_{\bR^n} f(x) \, \rho_4(\di x) \right). \nonumber
\end{align}
It is easy to see that the value in \eqref{eq-def-int} depends only on $\nu = \rho_1-\rho_2 + \ri (\rho_3 - \rho_4)$ and not on the specific choices of the $\rho_i$, see~\cite{B19,BKL21,LPS2018}. A complex L\'evy type measure is called a \emph{quasi-L\'evy type measure} if its imaginary part is 0, i.e.~if  $\rho_3$ and $\rho_4$ in the representation above can be chosen to be the zero-measure. (We prefer the name \lq quasi-L\'evy type measure\rq~rather than \lq real L\'evy type measure\rq~since this notation is then in line with \cite{B19,BKL21,LPS2018}.) \emph{Quasi-L\'evy type measures integrating $\min(1,|x|)$} are defined similarly.
\end{definition}

The next theorem provides a L\'evy--Khintchine type representation with complex L\'evy type measures for the characteristic function
of invertible complex measures on $\bR^n$. Observe that for each fixed $z\in \bR^n$, the function $\bR^n \to \bC$, $x\mapsto \re^{\ri z^T x} - 1$, satisfies
$|f(x)| \leq C \min (1,|x|)$
for some constant~$C$, hence the integral appearing in \eqref{eq-main2} below is defined.

\begin{theorem}[Charact. of invertible complex measures on $\bR^n$ via char. functions] \label{t-main3}
(a)
Let $n\in \bN$ and $\mu \in M(\bR^n)$ with characteristic function $\widehat{\mu}$. Then $\mu$ is invertible if and only if there are  $\nu_0\in M(\bR^n)$, $\gamma \in \bR^n$, $c \in \bC\setminus\{0\}$, $p\in \bN_0$, integers $m_1,\ldots, m_p$ and points $u_1,\ldots, u_p$ on the $(n-1)$-dimensional unit sphere $S^{n-1}$ in $\bR^n$, such that
    \begin{equation}\label{eq-main2}
    \widehat{\mu}(z) = c \, \exp\left \{ \ri \gamma^T z + \int_{\bR^n} \left( \re^{\ri z^T x} -1\right) \, ( \nu_0 + \nu_1) (\di x)\right\}
    \end{equation}
    for all $z \in \bR^n$, where the quasi-L\'evy type measure $\nu_1$ integrating $\min(1,|x|)$ has polar representation
    \begin{equation}\label{eq-main3}
   \nu_1(B) := \int_{S^{n-1}}  \int_0^\infty \one_B(r\xi) r^{-1} \re^{-r} \, \di r \, \Lambda(\di \xi), \quad B \in \cB^n_0,
   \end{equation}
    with the finite signed measure $\Lambda$ on $S^{n-1}$ being given by
   \begin{equation} \label{eq-main4} \Lambda :=
    \sum_{j=1}^p \left( m_j \delta_{u_j}^{S^{n-1}} - m_j \delta_{-u_j}^{S^{n-1}}\right)
   \end{equation}
   ($\delta_{u_j}^{S^{n-1}}$ being the Dirac measure at $u_j \in S^{n-1}$, considered as a measure on $S^{n-1}$ with its Borel sets).\\
(b) The constant $c$ in \eqref{eq-main2} is equal to $\mu(\bR^n)$ which is necessarily different from 0. The complex measure $\nu_0\in M(\bR^n)$ may be chosen to satisfy $\nu_0(\{0\})  = 0$.\\
(c) When $\mu$ is an invertible finite signed measure on $\bR^n$, then
the  $\nu_0\in M(\bR^n)$ in \eqref{eq-main2} may be chosen to be a finite signed measure with $\nu_0(\{0\}) = 0$ (and $c=\mu(\bR^n)$ is then real-valued as well).\\
(d) Let $n\in \bN$. A function $f:\bR^n \to \bC$ is the characteristic function of some invertible complex measure $\mu\in M(\bR^n)$ if and only if it is of the form specified by the right-hand side of \eqref{eq-main2} with $\nu_1$ specified in \eqref{eq-main3} and \eqref{eq-main4}.
\end{theorem}

In the above theorem, $\nu_1$ can be interpreted as the product measure of the discrete, finitely carried signed measure $\Lambda$ on the unit sphere $S^{n-1}$ and the measure on $(0,\infty)$ with Lebesgue density $r\mapsto r^{-1} \re^{-r}$, where $\bR^{n}\setminus \{0\}$ is identified with $S^{n-1} \times (0,\infty)$ (since $\re^{\ri z^T 0}-1=0$, the integral over $\bR^n$ is the same as the integral over $\bR^n \setminus \{0\}$).  The set function $\nu:= \nu_0 + \nu_1 : \cB^n_0 \to \bC$ defines a complex L\'evy type measure, with $\nu_0$ corresponding to a general finite complex part and $\nu_1$ taking care of the non-finite part, which is seen to have a very special structure.

As mentioned before, quasi-infinitely divisible distributions are the probability distributions whose characteristic function allows a L\'evy--Khintchine type representation with a quasi-L\'evy type measure rather than a L\'evy measure, see \cite[Thm.~2.2]{BKL21} and \cite[Rem.~2.4]{LPS2018}. Allowing complex L\'evy type measures in that representation does not enlarge the class of probability measures with this representation, see~\cite[Thm.~3.2]{B19} and~\cite[Thm.~2.7]{BKL21}, but it does when instead of probability measures general  complex measures with L\'evy--Khintchine type representation are allowed. In the following, we restrict ourselves to L\'evy--Khintchine type representations with complex L\'evy type measures integrating $\min(1,|x|)$, which is sufficient for our purposes.

\begin{definition} \label{d-qid}
A complex measure $\mu \in M(\bR^d)$ is said to have a \emph{L\'evy--Khintchine type representation with no Gaussian component and with a complex L\'evy type measure integrating $\min(1,|x|)$}, if  there are $c\in \bC \setminus \{0\}$, $\gamma \in \bR^n$ and a complex L\'evy type measure $\nu$ integrating $\min(1,|x|)$ such that
\begin{equation} \label{eq-CLK}
\widehat{\mu}(z) = c \, \exp \left( \ri \gamma^T z  + \int_{\bR^n}
\left( \re^{\ri z^T x } - 1 \right) \, \nu(\di x) \right)
\end{equation}
for all $z\in \bR^n$ (observe that the integrand can be bounded by a constant times $\min(1,|x|)$, hence the integral with respect to $\nu$ is well-defined).  The class of all $\mu \in M(\bR^n)$ with such a representation will be denoted by $\CLK$.
\end{definition}

\begin{remark} \label{r-qid-uniqueness}
(a) The subscript 0 in $\CLK$ gives reference to the fact that the complex L\'evy type measure integrates $\min(1,|x|)$, which for L\'evy measures corresponds to the fact that the extended truncation function $c_0(x):=0$ in the L\'evy--Khintchine representation can be used, see \cite[Prop.~9.18]{BrockLind2024} or \cite[Eq.~(8.7)]{Sato2013}.\\
(b) If $\mu \in \CLK$, then
the constant $c$ in \eqref{eq-CLK}  is unique and given by $c=\widehat{\mu}(0) = \mu(\bR^n) \neq 0$. However, also $\gamma$ and $\nu$ are unique. The latter fact follows from the proofs of Proposition~9.8 in \cite{BrockLind2024} or Theorem~8.1~(ii) in \cite{Sato2013}, which easily carry  over to complex L\'evy type measures integrating $\min(1,|x|)$.\\
(c) When $\mu \in \CLK$ is a probability measure, then $\nu$ is necessarily real-valued, i.e. a quasi-L\'evy type measure integrating $\min(1,|x|)$, see \cite[Thm.~3.2]{B19} and~\cite[Thm.~2.7]{BKL21}. The proof given there carries over to finite signed measures in $\CLK$ without change.\\
(d) A probability measure $\mu\in \CLK$ is quasi-infinitely divisible.\\
(e) Cuppens~\cite[Sect.~4.3]{Cuppens1975} calls probability measures on $\bR^d$ having a representation as in \eqref{eq-CLK} with a \emph{finite} signed measure $\nu$ (and some additional non-negative quadratic form being present) to have a De-Finetti representation, and he obtains several results for such probability measures.
\end{remark}

The following result is immediate from Theorem~\ref{t-main3}:

\begin{corollary} \label{c-QID}
Let $\mu \in M(\bR^n)$ be invertible. Then $\mu \in \CLK$.
\end{corollary}

Sections~\ref{S3}--\ref{S8} will be dedicated to the proof of the main results, i.e. to Theorem~\ref{t-main1}, Corollary~\ref{c-main2} and Theorem~\ref{t-main3}.

\section{Preliminaries and general notation} \label{S-prelim}

We denote $\bN=\{1,2,\ldots \}$ and $\bN_0 = \bN \cup \{0\}$. The zero-element in $\bR^n$ (all elements viewed as column vectors) will be denoted by either $0$ or $0_n$. The euclidian norm of an element $x\in \bR^n$ is denoted by $|x|$, and the $(n-1)$-dimensional unit sphere in $\bR^n$ by $S^{n-1}$.
For $A\subset X$, the indicator function of the set $A$  is denote by $\one_A$, hence $\one_A(x) = 1$ if $x\in A$ and $\one_A(x) = 0$ for $x\in X\setminus A$. If $X$ is a topological space, we denote the system of its Borel sets, i.e.~the $\sigma$-algebra generated by the open sets,  by $\cB(X)$. In particular, when $\tau$ is a topology on $\bR^n$, we write $\cB(\bR^n_\tau)$, whereas when we intend the euclidian topology we usually omit explicit mentioning of the topology. Hence $\cB^n = \cB(\bR^n)$. The discrete topology on a set $X$ is described by the power set of $X$, i.e.~the topology in which every subset of $X$ is open (and hence also closed). We denote the discrete topology by $X_d$, in particular $\bR^n_d$ is $\bR^n$ equipped with the discrete topology. The compact sets in $\bR^n_d$ are the finite subsets of $\bR^n$. If $X$ is a Hausdorff topological space, then a (positive) measure $\mu:\cB(X) \to [0,\infty]$ is called \emph{regular} if $\mu(K) < \infty$ for each compact subset $K\subset X$, if for every $A\in \cB(X)$ we have $\mu(A) = \inf \{\mu(U): U \supset A \mbox{ and $U$ open}\}$ and if for every open $U\subset X$ we have $\mu(U) = \sup \{ \mu(K) : K \subset U \mbox{ and $K$ compact}\}$. A complex or finite signed measure is called \emph{regular} if its total variation is regular. On $\bR^n$ equipped with the euclidian topology, every finite measure on $\cB^n$ is regular (e.g.~\cite[Prop.~7.2.3]{Cohn2013}), hence the elements of $M(\bR^n)$ are automatically regular. If $X$ and $Y$ are locally compact Hausdorff spaces, we denote by $C(X,Y)$ the set of all continuous functions from $X$ to $Y$, while if $Y=\bC$ with the euclidian topology, we usually simply write $C(X)$. We denote by $C_c(X)$ those elements in $C(X)$ that have compact support and by $C_0(X)$ those elements in $C(X)$ that vanish at infinity.

A \emph{locally compact abelian group} is an abelian group $G$ (the group operation written as addition) which is also a locally compact Hausdorff space, such that the group operations $G\times G \to G$, $(x,y) \to x+y$ (on the product space) and $G\to G$, $x\mapsto -x$, are continuous. A (positive) measure $H:\cB(G) \to [0,\infty]$ is called \emph{Haar measure} on $G$, if it is non-zero, regular and translation invariant, the latter meaning that $H(A+x) = H(A)$ for all $A\in \cB(G)$ and $x\in G$ (with $A+x := \{a+x : a \in A\}$). On every locally compact abelian group there exists a Haar measure $H$, which is unique up to multiplication by positive constants (e.g.~\cite[Thms.~9.2.2 and 9.2.6]{Cohn2013}), which is why we usually simple speak of \emph{the Haar measure}. On $\bR^n$, the Haar measure is Lebesgue measure,  henceforth denoted by $\lambda^n$, while on $\bR^n_d$ the Haar measure is the counting measure, which assigns to each $A\subset X$ its number of elements (see \cite[Ex.~9.2.1]{Cohn2013}), denoted henceforth by $\zeta^n$. Observe that the counting measure $\zeta^n$ is not $\sigma$-finite. For a given countable set $C\subset \bR^n$, we denote by $\zeta^n_C$ the measure on $\cB^n$ (!) that counts the number of elements of $A\cap C$ for $A\in \cB^n$, i.e. $\zeta^n_C = \sum_{c\in C} \delta_c^n$ (on $(\bR^n,\cB^n)$). This measure is $\sigma$-finite. For a (positive) measure $\mu$ we denote by $L^1(\mu)$ the set of (equivalence classes of) $\mu$-integrable $\bC$-valued functions, while if $G$ is a locally compact abelian group we write $L^1(G)$ instead of $L^1(H)$, where $H$ is the Haar measure on $G$. In particular, $L^1(\bR^n)$ is the set of all $\cB_n$-measurable, Lebesgue-integrable functions on $\bR^n$, while $L^1(\bR^n_d)$ is the set of all functions $f:\bR^n \to \bC$ that are integrable with respect to $\zeta^n$, i.e. which vanish outside a countable set and whose values on the countable set are absolutely summable.

If a complex, positive or finite signed measure $\nu$ has a density $f$ with respect to a positive measure $\mu$, we write $\nu = f \mu$, and if $\mu=\lambda^n$ is the Lebesgue measure on $\bR^n$, we shall also write $\nu(\di s) = f(s) \lambda^n(\di s)$ or $\nu(\di s) = f(s) \, \di s$ rather than $\nu = f\lambda^n$, and similarly for integration with respect to Lebesgue measure. Image measures under a measurable mapping $T$ will be denoted by $T(\mu)$, and if $U \in \bR^{n\times n}$ is an orthogonal matrix (the collection of all orthogonal $(n\times n)$-matrices is denoted by $O(n)$), we also write $U(\mu)$ to indicate the image measure of $\mu$ under the induced mapping $\bR^n \to \bR^n $, $z\mapsto Uz$. The product measure of two $\sigma$-finite positive or complex measures $\rho_1$ and $\rho_2$ is denoted by $\rho_1 \otimes \rho_2$.

The characteristic function, or Fourier transform, of a complex measure $\mu \in M(\bR^n)$, was already defined in \eqref{eq-def-FT}. For $f\in L^1(\bR^n)$, the Fourier transform is denoted by $\cF f:\bR^n \to \bC$ and given by
$$\cF f(z) := \int_{\bR^n} \re^{\ri z^T x} \, \di x, \quad z \in \bR^n,$$
so that $\cF f$ is the characteristic function of $f \lambda^n$.
The exponential of $\nu\in M(\bR^n)$ was defined in \eqref{eq-exponential}. The characteristic function of $\exp(\nu)$ is given by
$$\widehat{\exp(\nu)}(z) = \sum_{j=0}^\infty \frac{1}{j!} (\widehat{\nu}(z))^j  = \exp (\widehat{\nu}(z)), \quad z \in \bR^n,$$
hence $\exp(\nu_1+\nu_2) = \exp(\nu_1) \ast \exp(\nu_2)$ for $\nu_1,\nu_2 \in M(\bR^n)$.
In particular, if $\nu \in M(\bR^n)$ we can rewrite the characteristic function of $\exp(\nu)$ as
\begin{equation} \label{eq-exponential2}
\widehat{\exp(\nu)}(z) = \re^{\nu(\bR^n)}
\exp \left\{ \int_{\bR^n} \left( \re^{\ri z^T x} - 1 \right) \, \nu(\di x) \right\}, z \in \bR^n,
\end{equation}
showing that $\exp(\nu) \in \CLK$ whenever $\nu\in M(\bR^n)$, and in that case the complex L\'evy type measure is finite and given by $\nu$ (restricted to $\cB^n_0$).  Conversely, if a function $f:\bR^n \to \bC$ is of the form
\begin{equation} \label{eq-exponential3}
f(z) = c \, \exp \left\{ \ri \gamma^T z + \int_{\bR^n} \left( \re^{\ri z^T x} - 1 \right) \, \nu(\di x) \right\}, \quad z \in \bR^n,
\end{equation}
for some $\nu\in M(\bR^n)$ (i.e.~for some \emph{finite} complex L\'evy type measure),  $c\in \bC \setminus \{0\}$ and $\gamma\in \bR^n$,
then it follows from \eqref{eq-exponential2} that $f$ is the characteristic function of some $\mu\in M(\bR^n)$, namely of $\delta^n_\gamma \ast \exp (\nu + \alpha \delta_0^n)$, where $\alpha$ is some complex number such that $\exp(\alpha) = c/\re^{\nu(\bR^n)}$. The complex measure $\mu$ is then necessarily in $\CLK$. Hence, complex measures in $\CLK$ with \emph{finite} complex L\'evy type measure are exactly those complex measures that can be written as $\delta_\gamma^n \ast \exp(\nu)$ for some $\nu\in M(\bR^n)$ and $\gamma \in \bR^n$.

Finally, we recall the distinguished logarithm. Let $f:\bR^n \to \bC$ be a continuous function such that $f(0) = 1$ and $f(z) \neq 0$ for all $z\in \bR^n$. Then there exists a unique continuous function $\psi:\bR^n \to \bC$ with $\psi(0) = 0$ such that $\exp(\psi(z)) = f(z)$ for all $z\in \bR^d$, see  \cite[Thm.~8.21]{BrockLind2024} or \cite[Lem.~7.6]{Sato2013}. The function $\psi$ is called the \emph{distinguished logarithm} of $f$. Consequently, when $g\in C(\bR^n)$ with $g(z) \neq 0$ for all $z\in \bR^n$, then we can write $g(z) = g(0) \exp (\psi(z))$, where $\psi$ is the distinguished logarithm of $g/g(0)$.

All definitions regarding Banach algebras will be given in Section~\ref{S6}.

\section{Plan for the proofs of the main results and determination of the topologies $\tau_i$ and the $\mu_i$ appearing in Theorem~\ref{t-Taylor2}} \label{S3}

Our main results will follow from Theorem~\ref{t-Taylor2}. For its application, we must
\begin{enumerate}
\item determine all topologies $\tau$ on $\bR^n$ that are at least as strong as the euclidian topology, such that $\bR_\tau^n$ remains a locally compact abelian group,
\item  identify the Haar measure $H=H_\tau$ for theses topologies,
\item find the form of general complex measures $\widetilde{\mu_i}_{|\cB^n}$, where $\widetilde{\mu_i}$ is absolutely continuous with respect to the corresponding Haar measure,
\item and finally characterise the invertible complex measures of the form $\alpha \delta_0^n+  \widetilde{\mu_i}_{|\cB^n}$ with $\widetilde{\mu_i}$ from (iii).
\end{enumerate}
In this section we deal with  steps (i) to (iii), while Sections~\ref{S4}--\ref{S7} are concerned with step (iv). Luckily, the answer to (i) has already been given by Hewitt~\cite[Thm.~2.2]{H64}. Accordingly, if $\tau$ is a topology on $\bR^n$, at least as strong as the euclidian topology, such that $\bR_\tau^n$ remains a locally compact abelian group, then there is some $q\in \{0,1,\ldots, n\}$ such that $\bR_\tau^n$ is topologically isomorphic to $\bR^q \times \bR^{n-q}_d$, where $\bR^q$ is equipped with the euclidian topology, $\bR^{n-q}_d$ with the discrete topology  and $\bR^q \times \bR^{n-q}_d$ carries the product topology. Here, two locally compact abelian groups $G$ and $G'$ are called \emph{topologically isomorphic}, if there is a group isomorphism $U:G\to G'$ which is also a homeomorphism, which is then called a \emph{topological isomorphism}. Since later we will need image measures, we are also interested in the form of these topological isomorphisms. The answer to this question can be derived from the proof of Theorem~2.2 in Hewitt~\cite{H64}. To do that, let $\bR^n_\tau$ be a locally compact abelian group with topology $\tau$ as least as strong as the euclidian topology. Since the answer is trivial if $\tau$ is the euclidian or discrete topology, we exclude this case for the moment.  Then it is shown in \cite[proof of Thm.~2.2]{H64} that there exists a linear subspace $D\subset \bR^n$ of dimension $q\in \{1,\ldots, n-1\}$ and a linear bijection $\beta:\bR^q \to D$, where $D$ is open (and simultaneously closed) in $\tau$, such that $\beta$ defines a topological isomorphism from $\bR^q$ (equipped with the euclidian topology) onto $D$ (equipped with the relative topology of $\tau$); in particular, the relative topology of $\tau$ on $D$ coincides with the euclidian topology on $D$. Now let $x_1,\ldots, x_q$ be an orthonormal basis of $D$ (with respect to the euclidian topology) and denote $y_1:= \beta^{-1}(x_1),\ldots, y_q:= \beta^{-1} (x_q)$. Let $\gamma:\bR^q \to \bR^q$ be a linear mapping defined by $\gamma(e_i) := y_i$, $i=1,\ldots, q$, where $e_i$ denotes the $i^{\rm th}$ unit vector in $\bR^q$. Then $\beta \circ \gamma:\bR^q \to D$ is a topological isomorphism with $\beta \circ \gamma(e_i) = x_i$ for all $i \in\{1,\ldots, q\}$. Denote by $W \subset \bR^n$ the orthogonal complement of $D$ in $\bR^n$ (with respect to the euclidian topology) and let $x_{q+1}, \ldots, x_n$ be an orthonormal basis for~$W$. Define the linear map $U:\bR^n \to \bR^n$ by $U(e_i) = x_i$, $i=1,\ldots, n$, where $e_i$ is the $i^{\rm th}$ unit vector in $\bR^n$ and $\bR^q$ is identified with $\bR^{q} \times \{0_{n-q}\} \subset \bR^n$ (having the first $q$ unit vectors). Then $U$ is described by an orthogonal matrix (which we also denote by~$U$). Denote by $\tau'$ the topology on $\bR^n$ induced by $U$ and $\tau$, i.e.~a set $A\subset \bR^n$ is open with respect to $\tau'$ if and only if it is the preimage under $U$ of an open set with respect to $\tau$. We show that $\tau'$ equals the product topology of $\bR^q$ and $\bR^{n-q}_d$. To this end, observe that every $A\subset \bR^n$, which is open with respect to $\tau$, can be written as $A = \bigcup_{x\in W} ((A-x)\cap D)) + x$. Since $\bR^n_\tau$ is a topological group and $D$ is open, also $(A-x)\cap D$ is open with respect to $\tau$. Hence each open $A$ in $\tau$ can be written as a union $A = \bigcup_{x\in W} (B_x+x)$ with $B_x\subset D$ open (with respect to the relative topology of $\tau$ on $D$, hence with respect to $\tau$, since $D$ is open). On the other hand, if $B_x \subset D$ is open, then $B_x + x$ is open since $\bR^n_\tau$ is a topological group, hence so are arbitrary unions of such sets. We conclude that $\{ B+x : B \subset D \, \, \mbox{open}, x \in W\}$ is a basis for~$\tau$.
Since the relative topology of $\tau$ on $D$ coincides with the euclidian topology on $D$,
we conclude that the system $\mathcal{A}'$ of all sets of the form $B'+x'$, where $B'$ is open in $\bR^q$ and $x'\in \bR^{n-q}$ (with $\bR^q$ identified with $\bR^q \times \{0_{n-q}\}$ and $\bR^{n-q}$ identified with $\{0_q\} \times \bR^{n-q}$), is a basis for $\tau'$. On the other hand, the set $\mathcal{A'}$ is easily seen to be a basis for the product (topological) space $\bR^q \times \bR^{n-q}_d$. Summarising, we have:

\begin{lemma} \label{l-Hewitt}
Let $\tau$ be a topology on $\bR^n$, at least as strong as the euclidian topology, such that $\bR^n_\tau$ remains a locally compact abelian group. Then there exist $q\in \{0,\ldots, n\}$ and an orthogonal matrix $U\in \bR^{n\times n}$ such that the induced linear map $U:\bR^n \to \bR^n$ is a topological isomorphism from $\bR^q \times \bR^{n-q}_d$ onto $\bR^n_\tau$.
\end{lemma}

For the determination of the Haar measures $H_{\tau}$ appearing in Theorem~\ref{t-Taylor2}, suppose first that $U$ appearing in Lemma~\ref{l-Hewitt} is the identity matrix, so that $\bR^n_\tau = \bR^q \times \bR^{n-q}_d$ for some $q\in \{0,\ldots, n\}$. If $q=n$, we know already that the Haar measure is equal to $n$-dimensional Lebesgue measure $\lambda^n$ on $\cB^n$, while when  $q=0$, we have the discrete topology, and the Haar measure is the counting measure $\zeta^n$. When $q\in \{1,\ldots, n-1\}$, then the natural candidate for the Haar measure on $\bR^q \times \bR^{n-q}_d$ is the product measure $\lambda^q \otimes \zeta^{n-q}$. This is indeed the case, but for proving it, one has to be careful, since the counting measure is not $\sigma$-finite and hence one has to rely on the theory of regular measures on locally compact abelian groups and one should clarify first what is meant by the \lq product measure\rq, when one factor is not $\sigma$-finite. All necessary facts that we need can be found in the book by Cohn~\cite[Ch.~7 and 9]{Cohn2013}. The Haar measures $\lambda^q$ and $\zeta^{n-q}$ on $\bR^q$ and $\bR^{n-q}_d$, respectively, are regular. As shown in \cite[Prop.~7.6.4]{Cohn2013},
\begin{align*}
I  : &  C_c(\bR^q \times \bR^{n-q}_d,\bR)  \to  \bR ,\\
  & f \mapsto \int_{\bR^q} \int_{\bR^{n-q}_d} f(x,y) \, \zeta^{n-q}(\di y) \, \lambda^q(\di x) =
 \int_{\bR^{n-q}_d} \int_{\bR^q} f(x,y) \, \lambda^q(\di x) \, \zeta^{n-q} (\di y),
 \end{align*}
defines a positive linear functional on $C_c (\bR^q \times \bR^{n-q}_d,\bR)$, and indeed the order of integration may be interchanged. By the Riesz representation theorem, there exists a regular measure  $H$ on $\bR^q \times \bR^{n-q}_d$ corresponding to $I$, which Cohn~\cite[p.~222]{Cohn2013} calls the \emph{regular Borel product} of $\lambda^q$ and $\zeta^{n-q}$. We denote it by $\lambda^q \times \zeta^{n-q}$. Since $\lambda^q$ and $\zeta^{n-q}$ are translation invariant, so is the linear functional $I$ and hence so is the corresponding measure $\lambda^q \times \zeta^{n-q}$, which is hence the Haar measure on $\bR^q \times \bR^{n-q}_d$.

So far we have carried out steps (i) and (ii) of our road map (at least when $U$ is the identity matrix). Next, we have to find the form of all absolutely continuous finite complex measures $\widetilde{\mu}$ with respect to $\lambda^q \times \zeta^{n-q}$, when restricted to $\cB^n$. When $q=n$, it is clear from the standard form of the Radon-Nikodym theorem (observe that $\lambda^n$ is $\sigma$-finite) that we are dealing with complex measures of the form $\widetilde{\mu} (\di s) = f(s) \, \lambda^n(\di s)$ with $f:\bR^n \to \bC$ being $\cB^n$-measurable and Lebesgue-integrable. When $q\in \{0,\ldots, n-1\}$, $\sigma$-finiteness is lost, but the Radon-Nikodym theorem still holds, in the sense that a regular  complex measure $\widetilde{\mu}$ (which lives on $\cB(\bR^q \times \bR^{n-q}_d)$, not just on $\cB^n$!) is absolutely continuous with respect to $\lambda^q \times \zeta^{n-q}$ if and only if there is some $f\in L^1(\bR^q \times \bR^{n-q}_d)$ such that $\widetilde{\mu}$ has density~$f$ with respect to $\lambda^q \times \zeta^{n-q}$; see~\cite[Prop.~7.3.10]{Cohn2013}. To gather more information about $f$, we observe from~\cite[Lem.~9.4.2]{Cohn2013} that there is a sequence $(K_k)_{k\in \bN}$ of compact subsets of $\bR^q \times \bR^{n-q}_d$ such that $f(x,y) = 0$ for all $(x,y) \notin \bigcup_{k=1}^\infty K_k$. Now if $K$ is a compact subset of $\bR^q \times \bR^{n-q}_d$, then the projection of $K$ onto $\bR^{n-q}_d$ is compact, hence is finite, since $\bR^{n-q}_d$ carries the discrete topology. It follows that $K$ is contained in $\bR^{q} \times F$ for some finite set $F\subset \bR^{n-q}$, hence $f$ vanishes outside $\bR^{q} \times C$ for some countable set $C$. Denote by $f^c:\bR^q \to \bC$ the section of $f$ at $c\in C$, i.e. $f^c(x) = f(x,c)$ for $x\in \bR^q$. Then $f^c$ is $\cB^q$-measurable (cf.~\cite[Lem.~7.6.1]{Cohn2013}) and since $C$ is countable, the function $f$ is $\cB^n$-measurable. Further, by~\cite[Thm.~7.6.7]{Cohn2013},  $f^c \in L^1(\bR^q)$ for $\zeta^{n-q}$-almost all $c\in C$ (hence for all $c\in C$), and for all $A\in \cB(\bR^q \times \bR^{n-q}_d)$, in particular all $A\in \cB^n$,  we have
\begin{align*}
& \int_{\bR^q \times \bR^{n-q}_d} f(x,y)  \one_A (x,y) (\lambda^q \times \zeta^{n-q})(\di(x,y)) \\
&  = \int_{\bR^{n-q}_d} \int_{\bR^q} f(x,y) \one_A(x,y) \,
\lambda^q (\di x) \, \zeta^{n-q}(\di y)  = \sum_{c\in C} \int_{\bR^q} f(x,c) \one_A(x,c) \, \lambda^q(\di x),
\end{align*}
and the same is true when $f$ is replaced by $|f|$.
With  $\zeta^{n-q}_C$ as defined in Section~\ref{S-prelim} (which is $\sigma$-finite), we easily see that the right-hand side of the last equation is equal to
$\int_A f(x,y) (\lambda^q \otimes \zeta^{n-q}_C) \, \di(x,y)$
for all $A\in \cB^n$. Summing up, we conclude that whenever the regular  complex measure $\widetilde{\mu}$ is absolutely continuous with respect to $\lambda^q \times \zeta^{n-q}$, then there are a countable set $C\subset \bR^{n-q}$ and a $\cB^n$-measurable function $f:\bR^n \to \bC$ such that $\widetilde{\mu}_{|\cB^n}$ is equal to the complex  measure $f\, \lambda^q \otimes \zeta^{n-q}_C$ on $\cB^n$, i.e.~the complex measure on $\cB^n$ with density~$f$ with respect to $\lambda^q \otimes \zeta^{n-q}_C$. We further have $\sum_{c\in C} \int_{\bR^q} |f(x,c)| \lambda^q(\di x) < \infty$.

If we drop the assumption that $U$ is the identity matrix, where $U$ is the orthogonal matrix that induces the topological isomorphism from $\bR^q \times \bR^{n-q}_d$ onto $\bR^n_\tau$ according to Lemma~\ref{l-Hewitt}, then it is easily seen that the image measure $U(\lambda^q \times \zeta^{n-q})$ is the Haar measure on $\bR^n_\tau$, that the regular complex measures on $\bR^n_\tau$ that are absolutely continuous with respect to $U(\lambda^q \times \zeta^{n-q})$ are exactly the image measures under $U$ of the corresponding ones on $\bR^q \times \bR^{n-q}$, and that also their restriction to $\cB^n$ is carried to the corresponding restriction. Combining this with Theorem~\ref{t-Taylor2}, with the previous notations, we immediately obtain the following result:

\begin{proposition} \label{p-Taylor3}
Let $\mu \in M(\bR^n)$ be invertible. Then there exist $p\in \bN_0$, $\nu \in M(\bR^n)$,  orthogonal matrices $U_1,\ldots, U_p \in \bR^{n\times n}$, integers $q_1,\ldots, q_p \in \{0,\ldots, n\}$, non-empty countable sets $C_1\subset \bR^{n-q_1}, \ldots, C_p \subset \bR^{n-q_p}$, complex numbers $\alpha_1,\ldots, \alpha_p$ and $\cB^n$-measurable functions $f_i:\bR^n \to \bC$, $i=1,\ldots, q$, each $f_i$ vanishing outside $\bR^{q_i} \times C_i$ and being integrable with respect to $\lambda^{q_i} \otimes \zeta_{C_i}^{n-q_i}$, such that
$$\mu = \mu_1 \ast \ldots \ast \mu_p \ast \exp(\nu),$$
where $\mu_i$ is of the form $\mu_i = U_i( \alpha_i \delta_0^n + f_i \lambda^{q_i} \otimes \zeta_{C_i}^{n-q_i})$ and is invertible for each $i\in \{1,\ldots, p\}$. Here, when $q_i = n$, we interpret $C_i$ as $\{0\}$ and $f_i \lambda^{q_i} \otimes \zeta_{C_i}^{n-q_i}$ as $f_i \lambda^n$.
\end{proposition}

With Proposition~\ref{p-Taylor3} we have completed steps (i) -- (iii) of our general plan and it remains to characterise when the complex measures $\mu_i$ of the special form appearing there are invertible. This will be done in the following sections, separately when $q_i=n$, $q_i=0$ and $q_i \in \{1,\ldots, n-1\}$. We close this section by pointing out that when complex measures of these specific forms are invertible, then their inverses have the same form. This result is again an immediate consequence of results of Taylor~\cite[Thm.~3.3 with Prop.~4.1]{T-spectrum}, since the complex measures of the form
$\alpha \delta_0^n + f \lambda^{q} \otimes \zeta_{C}^{n-q}$ for fixed $q\in \{0,\ldots, n\}$ form a closed subalgebra of $M(\bR^n)$, and for an  element $\mu$ of this subalgebra, the spectrum with respect to $M(\bR^n)$ is the same as the spectrum with respect to this subalgebra. This is by no means trivial, but true for the specific form, as follows easily from \cite[Thm.~3.3 with Prop.~4.1]{T-spectrum}, observing that we have already determined all possible topologies and hence the space $M_1(\bR^n)$ appearing in \cite[Prop.~4.1]{T-spectrum}.

\begin{proposition} {\rm (Taylor~\cite[Thm.~3.3 with Prop.~4.1]{T-spectrum}).} \label{p-Taylor4}
Let $q\in \{0,\ldots, n\}$, $U\in \bR^{n\times n}$ be an orthogonal matrix and
$\mu \in M(\bR^n)$ be invertible and of the form $\mu = U(\alpha \delta_0^n + f \lambda^q \otimes \zeta_C^{n-q})$, with
$\alpha \in \bC$, $C\subset \bR^{n-q}$ countable and $f$ vanishing outside $\bR^q \times C$ and being integrable with respect to $\lambda^q \otimes \zeta_C^{n-q}$. Then the inverse of $\mu$ is of the same form, i.e. can be written as $U (\widetilde{\alpha} \delta_0^n + \widetilde{f} \lambda^q \otimes \zeta_{\widetilde{C}}^{n-q})$ with $\widetilde{\alpha} \in \bC$, $\widetilde{C} \subset \bR^{n-q}$ being countable, and $\widetilde{f}$ vanishing outside $\bR^q \times \widetilde{C}$ and being integrable with respect to $\lambda^q \otimes \zeta^{n-q}_{\widetilde{C}}$.
\end{proposition}

\section{The case $q=n$} \label{S4}

In this section we characterise when a complex measure of the form $\mu= \alpha \delta_0^n + f \lambda^n$ is invertible, where $\alpha\in \bC$ and $f\in L^1(\bR^n)$; this corresponds to the case $q_i=n$ in Proposition~\ref{p-Taylor3}.   It turns out that there is a subtle difference between $n=1$ and $n\geq 2$ in the form of the L\'evy--Khintchine representation, which basically is due to the fact that $\{ z \in \bR^n : |x|>R\}$ for $R>0$ is connected when $n\geq 2$, while it is not when $n=1$. Let us first start with some general equivalences which are true for all dimensions and follow easily from results in Taylor~\cite{T73}:

\begin{proposition} \label{p1-q=n}
Let $n\in \bN$ and $\mu\in M(\bR^n)$ be of the form $\mu = \alpha \delta_0^n + f \lambda^n$ with
$\alpha\in \bC$ and $f\in L^1(\bR^n)$. Then the following are equivalent:
\begin{enumerate}
\item[(i)] $\mu$ is invertible.
\item[(ii)] $\alpha \neq 0$ and $\widehat{\mu}(z) \neq 0$ for all $z\in \bR^n$.
\item[(iii)] $\inf_{z\in \bR^n} |\widehat{\mu}(z)| > 0$, i.e.~\eqref{eq-bounded-cf} holds true.
\end{enumerate}
\end{proposition}

\begin{proof}
The equivalence of (ii) and (iii) is an easy consequence of the continuity of $\widehat{\mu}$ and the Riemann--Lebesgue lemma, according to which $\lim_{|z|\to \infty} \widehat{\mu}(z) = \alpha$. That (i) implies (iii) was already observed in the introduction. For the proof that (iii) implies (i), observe that if \eqref{eq-bounded-cf} holds, then by Taylor~\cite[Prop.~1.2.5]{T73}, the convolution operator $C_\mu:L^1(\bR^n) \to L^1(\bR^n)$, defined by $C_\mu f = \mu \ast f$, is invertible, hence $\mu$ is invertible by~\cite[Cor.~1.2.8]{T73}.
\end{proof}

Let us now add some more equivalences when $n=1$ and $\alpha \neq 0$, which were given by Taylor~\cite{T73} and Berger~\cite{B19}.

\begin{proposition} \label{p2-q=n}
Let $\mu \in M(\bR)$ be of the form $\mu = \alpha \delta_0^1 + f \lambda^1$ with $\alpha \in \bC\setminus \{0\}$ and $f\in L^1(\bR)$. Then the following are equivalent:
\begin{enumerate}
\item[(i)] $\mu$ is invertible.
\item[(iv)] $\mu \in \CLKeins$.
\item[(v)] $\mu \in \CLKeins$ with $\gamma=0$ in \eqref{eq-CLK} and complex L\'evy type measure $\nu$ being of the form
$$\nu(\di x) = (m \, \re^{-|x|}/x + g(x)) \lambda^1(\di x)$$
for some $m\in \bZ$ and $g\in L^1(\bR)$.
\item[(vi)] There exist $m \in \bZ$ and $\nu'\in M(\bR^1)$ such that $\mu = \sigma^{\ast m} \ast \exp(\nu')$, where $\sigma$ is the signed measure on $\bR$ defined in Theorem~\ref{t-Taylor1}
\end{enumerate}
\end{proposition}

\begin{proof}
Recall the equivalence of (i) and  (ii) in Proposition~\ref{p1-q=n}. By \cite[Thm.~4.4]{B19}, (ii) implies (v), which in turn implies (iv), and that (iv) implies (ii) is clear. That (i) implies (vi) is proved in~\cite[Sect.~1.6.4]{T73} (see also Sect.~8.4.1 in~\cite{T73}, where this is more clearly stated), and that (vi) implies (i) is clear since $\exp(\nu')$ and $\sigma$ are both invertible.
\end{proof}

Observe that when $\alpha = 0$, then (iv) does not imply (i). A counterexample is given by the exponential probability distribution with parameter 1, which satisfies (iv) with L\'evy measure $\nu(\di x) = x^{-1} \re^{-x} \one_{(0,\infty)}(x)$ (see \cite[Ex.~9.16 (e)]{BrockLind2024} or \cite[Ex.~8.10]{Sato2013}), but is not invertible, being an absolutely continuous distribution. Condition (vi) still implies~(i) and since there is no absolutely continuous invertible complex measure, it follows that there is no absolutely continuous complex measure which has a representation as in (vi).

We now give the counterpart to Proposition~\ref{p2-q=n} when $n\geq 2$. While invertibility of $\mu$ is still equivalent to $\mu \in \CLK$, the form of the complex L\'evy type measure $\nu$ in (v) and the representation in (vi) become simpler. More precisely, we have:

\begin{theorem} \label{t-delta-lebesgue1}
Let $n\geq 2$ and  $\mu\in M(\mathbb{R}^n)$ be of the form $\mu=\alpha\delta_0^n+f\lambda^n$, where $\alpha \in \bC\setminus\{0\}$ and $f\in L^1(\mathbb{R}^n)$. Then the following are equivalent:
\begin{enumerate}
\item[(i)] $\mu$ is invertible.
\item[(iv)] $\mu \in \CLK$.
\item[(v')] $\mu \in \CLK$ with $\gamma=0$ in \eqref{eq-CLK} and the complex L\'evy type measure $\nu$ being of the form
$\nu(\di x) = g(x) \lambda^1(\di x)$ for some $g\in L^1(\bR)$, i.e.~$\nu$ is finite (more precisely, can be extended to a finite complex measure on $\cB^n$) and absolutely continuous with respect to $\lambda^n$.
\item[(vi')] There exists $\nu' \in M(\bR^n)$ such that $\mu = \exp(\nu')$.
\end{enumerate}
\end{theorem}

\begin{proof}
That (v') implies (iv) is clear. Further, (iv) implies (ii) of Proposition~\ref{p1-q=n} since we assumed $\alpha\neq 0$, hence (iv) implies (i). The proof that (i) implies (v') is the hardest part. For its proof, assume that $\mu$ is invertible, so that  $\widehat{\mu}(z) \neq 0$ for all $z\in \bR^n$.
Replacing $\mu$ by $(\widehat{\mu}(0))^{-1} \mu$, we can assume without loss of generality that $\widehat{\mu}(0) = 1$. Since $\widehat{\mu}$ has no zeroes on $\bR^n$, we can take the distinguished logarithm $\psi:\bR^n \to \bC$ of~$\widehat{\mu}$ (see Section~\ref{S-prelim}).
The idea is to prove that $\psi$ has a limit at infinity, which implies that there exists a constant $c\in \mathbb{C}$ such that $\psi-c\in C_0(\mathbb{R}^n)$. Using then the denseness of the range of the $L^1$-Fourier transform in  $C_0(\mathbb{R}^n)$, we can apply a Wiener-L\'evy theorem.

To carry out the above idea, observe that $\re^{\psi(z)}=\widehat{\mu}(z)$ tends to $\alpha \neq 0$ as $|z|\to\infty$ by the Riemann-Lebesgue lemma. Denote by $\log \alpha$ some arbitrary but fixed complex logarithm of $\alpha$. Then $\re^{\psi(z) - \log \alpha}$ tends to 1 as $|z|\to\infty$. Consequently, for each $\varepsilon > 0$, there exists $R_\varepsilon > 0$ such that
$$\psi(z) - \log \alpha \in \bigcup_{k \in \bZ} \{ w + 2 \pi \ri k : w \in \bC, |w| < \varepsilon\}
$$
whenever $|z|>R_\varepsilon$. Since $z\mapsto \psi(z) - \log \alpha$ is continuous, since the set $\{ z \in \bR^n : |z| > R_\varepsilon\}$ is connected as a consequence of $n\geq 2$, and since the sets $\{ w +2 \pi \ri k : w \in \bC, |w| < \varepsilon\}$, $k\in \bZ$, are separated from each other when $\varepsilon < \pi$, we conclude that there is $k_0 \in \bZ$ (independent of $\varepsilon$) such that $\psi(z) - \log \alpha \in \{ w + 2 \pi \ri k_0 : |w|< \varepsilon\}$ whenever $\varepsilon \in (0,\pi)$ and $|z| > R_\varepsilon$. Hence $\lim_{|z|\to\infty} \psi(z)$ exists and is equal to $2 \pi \ri k_0 + \log \alpha$.

We conclude that
\begin{align*}
\widetilde{\psi}:\mathbb{R}^n\to \mathbb{C},\,\widetilde{\psi}(z):=\psi(z)-2\pi \ri k_0-\log \alpha
\end{align*}
is in $C_0(\mathbb{R}^n)$ (the space of continuous functions on $\bR^n$ vanishing at infinity, see Section~\ref{S-prelim}). As the range of the $L^1$-Fourier transform $\cF$ (cf. Section~\ref{S-prelim}) is dense in  $C_0(\mathbb{R}^n)$ with respect to uniform convergence (since it contains all Schwartz functions and since the set of Schwartz functions is dense),  there exists a function $h\in L^1(\mathbb{R}^n)$ such that
\begin{align} \label{eq-mu-1}
\sup_{z \in \bR^n} |\widetilde{\psi}(z)-\cF{h}(z)|<\ \log \frac{3}{2}.
\end{align}
Write
\begin{align}
\widehat{\mu}(z)
=&\re^{\widetilde{\psi}(z)-\cF{h}(z)}\, \re^{\cF{h}(z)} \, \re^{2\pi \ri k_0+\log \alpha}. \label{eq-mu-2}
\end{align}
Since
\begin{align*}
\re^{\widetilde{\psi}(z) - \cF{h}(z)} & = \frac{1}{\alpha} \widehat{\mu}(z) \re^{-\cF{h}(z)}
= \frac{1}{\alpha} \widehat{\mu}(z) \sum_{j=0}^\infty (\widehat{(-h \lambda^n)}(z))^j / j! = \frac{1}{\alpha} \widehat{\mu}(z) \, {[\exp (-h \lambda^n)]}^{\wedge}(z)
\end{align*}
we conclude that $\re^{\widetilde{\psi} - \cF{h}}$ is the characteristic function  of the complex measure
$$\rho := \frac{1}{\alpha} \mu \ast \exp(-h\lambda^n) =  (\delta_0^n + \alpha^{-1} f \lambda^n) \ast (\delta_0^n + \sum_{j=1}^\infty (-h \lambda^n)^{\ast j} / j!),$$
which is of the form $\rho = \delta_0^n + \varphi \lambda^n$ with $\varphi \in L^1(\bR^n)$. From \eqref{eq-mu-1} we derive
\begin{align*}
|\cF \varphi(z)| & =| \widehat{\rho}(z) - 1| =| \re^{\widetilde{\psi}(z) - \cF{h}(z)} - 1| \leq \re^{|\widetilde{\psi}(z)-\cF{h}(z)|}- 1 < \frac12
\end{align*}
for all $z\in \bR^n$. Define $W:= \{ w \in \bC : |w| < 1\}$ and the function $F: W \to \bC$, $w\mapsto \log (1+w)$, where $\log$ denotes the principal branch of the complex logarithm. Since $F$ is holomorphic on the open set $W$ with $F(0) = 0$ and $W$ contains the closure of $\{ \cF \varphi(z) : z \in \bR^n\}$, it follows from the Wiener--L\'evy theorem for Fourier transforms on $\bR^d$ (e.g.~Rudin~\cite[Theorem 6.2.4]{Rudin1962})
that $F(\cF \varphi)$ is in the range of the $L^1$-Fourier transform. In other words, there is $\tau \in L^1(\bR^n)$ such that
$$\cF \tau(z) = F(\cF \varphi(z)) = \log (1+\cF \varphi(z)) = \log (\widehat{\rho}(z))$$
for all $z\in \bR^n$, hence
$$\re^{\widetilde{\psi}(z) - \cF{h}(z)} = \widehat{\rho}(z) = \re^{\cF{\tau}(z)}.$$
Together with \eqref{eq-mu-2} we conclude that
\begin{align*}
\widehat{\mu}(z)=& \, \alpha\, \exp\left(\,\int_{\mathbb{R}^n}\re^{\ri y^T z}(\tau(y)+h(y))\, \di y\right).
\end{align*}
Specialising to $z=0$ gives $1= \alpha \, \exp\left( \int_{\bR^n} (\tau(y) + h(y)) \, \di y \right)$,
yielding (v') with $g=\tau+h$.

We have proved the equivalence of (i), (iv) and (v'). That (vi') implies (i) is clear from the discussion in the introduction. Conversely, as pointed out in the preliminaries following Equation~\eqref{eq-exponential3}, (v') implies that $\mu = \exp(g \lambda^1 + \beta \delta_0^n)$ with some $\beta\in \bC$, giving (vi).
\end{proof}

\begin{remark} \label{r-p=n}
When $n\in \bN$ and $\mu\in M(\bR^n)$ is invertible and signed, then the function $g$ appearing in (v) of Proposition~\ref{p2-q=n} and (v') of Theorem~\ref{t-delta-lebesgue1} is real-valued, as a consequence of Remark~\ref{r-qid-uniqueness}~(c).
\end{remark}

For properties of multivariate distributions, one is often interested in the question whether the property of interest is already determined by the corresponding property of its projections onto all real lines through the origin. Such results are typically called \lq Cram\'er--Wold devices\rq. For complex measures of the form $\alpha \delta_0^n + f \lambda^n$ with $\alpha \neq 0$ we have such a result for the properties of being invertible or of belonging to $\CLK$:

\begin{corollary}[Cram\'er--Wold device for invertibility and for $\mathrm{CLK}_0$] \label{c-Cramer-Wold}
Let $n\geq 2$ and $\mu\in M(\bR^n)$ be of the form $\mu = \alpha \delta_0^n + f \lambda^n$, where $\alpha\in \bC \setminus \{0\}$ and $f\in L^1(\bR^n)$. For each $a\in S^{n-1}$  denote by $\pi_a:\bR^n \to \bR a$ the projection onto the line $\bR a$, which is identified with $\bR$, and denote by $\pi_a(\mu)$ the image measure of $\mu$ under $\pi_a$.  Then the following are equivalent:
\begin{enumerate}
\item[(i)] $\mu$ is invertible.
\item[(iv)] $\mu \in \CLK$.
\item[(vii)]  $\pi_a(\mu)$ is invertible for all $a\in S^{n-1}$.
\item[(viii)] $\pi_a(\mu) \in \CLKeins$ for all $a\in S^{n-1}$.
\end{enumerate}
\end{corollary}

\begin{proof}
For each $a\in S^{n-1}$ the complex measure $\pi_a(\mu)$ is also of the form $\alpha \delta_0^1 + f_a \lambda^1$ for some $f_a \in L^1(\bR)$. The characteristic function of $\pi_a(\mu)$ is given by $\widehat{\pi_a(\mu)}(z) = \widehat{\mu}(za)$, $z\in \bR$. Hence $\widehat{\mu}$ is non-vanishing if and only if all $\widehat{\pi_a(\mu)}$ are non-vanishing. The claim then follows from
Propositions~\ref{p1-q=n}, \ref{p2-q=n} and Theorem~\ref{t-delta-lebesgue1}.
\end{proof}

\section{The case $q=0$} \label{S5}

When $\mu_i$ in Proposition~\ref{p-Taylor3} corresponds to the case $q_i=0$, then $\mu_i$ is of the form $\mu_i = U_i (\alpha_i \delta_0^n + f_i \zeta_{C_i}^n)$ for some countable set $C_i$, which just means that $\mu_i$ is a discrete complex measure on $(\bR^n,\cB^n)$. A characterisation when discrete complex measures are invertible, and when they are in $\CLK$, has been obtained
by Alexeev and Khartov~\cite[Thm. 3.2]{AlexeevKhartov23}, and in the following statement we also add some characterisation given in Taylor~\cite[Cor.~8.3.3]{T73} to it.

\begin{theorem} {\rm (Alexeev and Khartov~\cite[Thm.~3.2]{AlexeevKhartov23}, Taylor~\cite[Cor.~8.3.3]{T73}).} \label{t-Khartov}
Let $\mu \in M(\bR^n)$ be discrete. Then the following are equivalent:
\begin{enumerate}
\item[(i)] $\mu$ is invertible.
\item[(ii)] $\inf_{z\in \bR^n} |\widehat{\mu}(z)| > 0$.
\item[(iii)] There exists a discrete complex measure $\mu' \in M(\bR^n)$ such that $\mu' \ast \mu = \delta_0^n$.
\item[(iv)] $\mu \in \CLK$.
\item[(v)]  $\mu \in \CLK$ and the complex L\'evy type measure $\nu$ is finite and discrete.
\item[(vi)] There exist $\gamma \in \bR^n$ and $\nu'\in M(\bR^n)$ such that $\mu = \delta_\gamma^n \ast \exp(\nu')$.
\item[(vii)] There exist $\gamma \in \bR^n$ and a discrete complex measure $\nu' \in M(\bR^n)$ such that $\mu=\delta_\gamma^n \ast \exp(\nu')$.
\end{enumerate}
\end{theorem}
As stated above, the equivalence of (ii) -- (v) can be found directly in \cite[Thm.~3.2]{AlexeevKhartov23}. Further, (iii) obviously implies (i), which in turn implies (ii) as mentioned in the introduction, so that (i) -- (v) are indeed equivalent.  That (i) implies (vii) follows from Corollary~8.3.3 in Taylor~\cite{T73}, and that (vii) implies (vi) which in turn implies (i) is clear, so that all assertions are indeed equivalent. A different proof of the equivalence of  (v) and  (vii) can be obtained from the discussion in the preliminaries following Equation~\eqref{eq-exponential3}.

Actually, in \cite[Thm. 3.2]{AlexeevKhartov23} some further equivalent conditions are given and the countable set carrying  $\nu$ in (vi) is further specified, but for our purpose the statement as given here is quite sufficient.

Alexeev and Khartov~\cite[Thm.~2.2]{AlexeevKhartov23} also obtain a Cram\'er-Wold device for invertibility and for membership in $\mathrm{CLK}_0$ for discrete distributions. The result in \cite{AlexeevKhartov23} is stated for discrete probability measures $\mu$, but the proof given there, which shows that for discrete probability measures,  $\inf_{z\in \bR^n} |\widehat{\mu}(z)| > 0$ if and only if $\inf_{x\in \bR} |\widehat{\mu}(xa)| > 0$ for all $a\in S^{n-1}$,
 carries over to discrete complex measures without changes. We shall not need that result in the following, but state it for completeness.

\begin{theorem} {\rm (Cram\'er--Wold device for invertibility and for $\mathrm{CLK}_0$, cf. Alexeev and Khartov~\cite[Thm.~2.2]{AlexeevKhartov23}).} \label{t-Khartov-Cramer-Wold}
Let $n\geq 2$ and $\mu \in M(\bR^n)$ be discrete and denote by $\pi_a:\bR^n \to \bR a$ the projection onto the line $\bR a$, which is identified with $\bR$, and denote by $\pi_a(\mu)$ the image measure of $\mu$ under $\pi_a$. Then the following are equivalent:
\begin{enumerate}
\item[(i)] $\mu$ is invertible.
\item[(iv)] $\mu \in \CLK$.
\item[(viii)] $\pi_a(\mu)$ is invertible for all $a\in S^{n-1}$.
\item[(ix)] $\pi_a(\mu) \in \CLKeins$ for all $a\in S^{n-1}$.
\end{enumerate}
\end{theorem}

\section{Some results on Banach algebra valued functions} \label{S6}

\subsection{Motivation} \label{S-motivation}
Denote by $\GFSC$ the set of all functions $f:\bR^n \to \bC$ that can be written in the form  $f(z) = \sum_{x\in B} \re^{\ri x^T z} a_x$ for some countable set $B$ and an absolutely summable sequence $(a_x)_{x\in B}$ of complex numbers; the notion \lq GFS\rq~stands for \lq generalized Fourier series\rq. It is easily seen that $\GFSC$ is exactly the class of characteristic functions of discrete complex measures on $(\bR^n,\cB^n)$.
Then a discrete complex measure $\sum_{x\in B} a_x \delta_x$, with characteristic function $f$, is invertible in $M(\bR^n)$ if and only if there is some $f' \in \GFSC$ such that $f'(z) f(z) = 1$ for all $z\in \bR^n$; this is clear since the inverse of a discrete complex measure, if existent, must be discrete (see also Proposition~\ref{p-Taylor4}). For a function $f\in \GFSC$, the equivalent conditions (i), (ii) and (v) of Theorem~\ref{t-Khartov} can then be restated as:
\begin{enumerate}
\item[(i)] There exists $f'\in \GFSC$ such that $f'(z) f(z) = 1$ for all $z\in \bR^n$.
\item[(ii)] $f(z) \neq 0$ for all $z\in \bR^n$ and $\sup_{z\in \bR^n} |f(z)^{-1}| < \infty$.
\item[(v)] There exist $\gamma\in \bR^n$ and $g\in \GFSC$ such that
\begin{equation} \label{eq-7-a}
f(z) = f(0) \, \exp \left( \ri \gamma^T z + g(z) - g(0) \right) \quad \forall\; z \in \bR^n
\end{equation}
($g$ is the characteristic function of $\nu$ appearing in Theorem~\ref{t-Khartov}~(v)).
\end{enumerate}
Now if $\mu \in M(\bR^n)$ is of the form $\mu = \alpha \delta_0^n + h \, \lambda^q \otimes \zeta_C^{n-q}$ as in Proposition~\ref{p-Taylor3} with $q\in \{1,\ldots, n-1\}$ ($h$ taking the role of $f_i$ for the moment), then it will be shown in~\eqref{eq-mu-drei} below that the characteristic function $\widehat{\mu}$ is of the form $\widehat{\mu}(\cdot, z_2) = \sum_{y\in C\cup \{0\}} \re^{\ri z_2^T y} a_y$ for  an absolutely summable sequence $(a_y)_{y\in C\cup \{0\}}$ in a suitable Banach algebra $A$. Denoting the set of all such functions by $\GFSQ$, invertibility of $\mu$ will then mean that we can find some $f' \in \GFSQ$ such that $f'(z_2) \widehat{\mu}(\cdot, z_2)$ is equal to $\widehat{\delta_0^q} = 1$ for all $z_2\in \bR^{n-q}$, and we will be looking for a representation as in \eqref{eq-7-a} for $\widehat{\mu}(\cdot, z_2)$. This will be achieved in Theorem~\ref{t-characterisation}, which will then be applied in Section~\ref{S7} for the treatment of the case $q\in \{1,\ldots, n-1\}$. For the proof of Theorem~\ref{t-characterisation} we
establish some results on Banach algebra valued functions which may be interesting in their own right.

\subsection{Generalities on Banach algebras} \label{S-Banach-general}

In the following, all Banach algebras will be complex Banach algebras. We refer to the monographs by Allan~\cite{Allan} or Kaniuth~\cite{Kaniuth} for general information regarding Banach algebras, but shortly recall the necessary definitions and facts about them.  A \emph{commutative Banach algebra} is a complex Banach space $A$ together with a multiplication $A\times A \to A$, $(a,b) \mapsto ab$, such that $a(bc) = (ab)c$, $ab=ba$, $\alpha (ab)  =(\alpha a) b$, $a(b+c) = ab + ac$ and such that $\|ab\|_A \leq \|a\|_A \, \|b\|_A$ for all $a,b,c\in A$ and $\alpha \in \bC$, where $\|\cdot\|_A$ denotes the norm on the Banach space. The zero-element in $A$ is denoted by either $0$ or $0_A$. If there exists an element $e_A$ in $A$ such that $a e_A =a$ for all $a\in A$ and $\|e_A\|_A = 1$, then~$A$ is called \emph{unital}. The element $e_A$ is called the \emph{identity} of $A$ and then necessarily unique. An element $a\in A$ in a commutative unital Banach algebra is \emph{invertible} if there exists some $a^{-1}\in A$ such that $a^{-1} a  = e_A$. The element $a^{-1}$ is the \emph{inverse} of $a$ which is then necessarily unique. We denote the set of all invertible elements in $A$ by~$A^{-1}$. The spectrum $\mbox{Sp}_A (a)$ of $a\in A$ is the set of all $\alpha \in \bC$ such that $a-\alpha e_A$ is not invertible. In a commutative unital Banach algebra, the \emph{exponential} $\exp (a)$ of some element $a\in A$ is defined as
$\exp(a) := \sum_{j=0}^\infty \frac{1}{j!} a^j$,
where $a^0 := e_A$. Since $A$ is commutative, we have $\exp(a + b) = \exp(a) \exp(b)$ for $a,b\in A$. In particular, $\exp(a)$ is invertible with inverse $\exp(-a)$, so that $\exp(A) := \{ \exp(a) : a \in A\} \subset A^{-1}$. The set $A^{-1} \subset A$ is a group under multiplication and carries the relative topology of $A$ in $A^{-1}$. It is known that an element $b\in A$ is in $\exp(A)$ if and only if it is in the connected component of $A^{-1}$ that contains $e_A$, see e.g.~Allan~\cite[Theorem 4.105]{Allan}, Kaniuth~\cite[Theorem 3.2.6]{Kaniuth} or Taylor~\cite[Section 1.6.1]{T72}. An element $b\in A$ is said to \emph{have a logarithm in $A$} if there exists some $a\in A$ such that $\exp(a) = b$, i.e. if $b\in \exp(A)$, and any such number $a$ is then called a \emph{logarithm} of $b$.
Let us recall some examples of Banach algebras which are relevant for us.

\begin{example} \label{ex-Banach-algebra-1}
(a) The complex numbers $\bC$, with the usual addition, multiplication and norm, clearly form a unital commutative Banach algebra.\\
(b) The set $M(\bR^n)$ with the usual addition, multiplication by scalars, convolution as multiplication and the total variation norm as norm is a  commutative unital Banach algebra with identity $\delta_0^n$ (e.g.~Cohn~\cite[Props.~4.1.8, 9.4.6]{Cohn2013}; actually, this is true more generally for the space of regular complex measures on locally compact abelian groups). The notions of invertibility and of the exponential of $\mu \in M(\bR^n)$ introduced in the introduction coincide with the corresponding notions just given for the Banach algebra $M(\bR^n)$.\\
(c) Let $L(\bR^n)$ be the set of all $\mu \in M(\bR^n)$ that are absolutely continuous with respect to $n$-dimensional Lebesgue measure $\lambda^n$ and $\bC \delta_0^n + L(\bR^n)$ the set of all $\nu \in M(\bR^n)$ that can be written as $\alpha \delta_0^n + \mu$ with $\alpha \in \bC$ and $\mu \in L(\bR^n)$. By the Radon-Nikodym theorem, $L(\bR^n)$ can be identified with $L^1(\bR^n)$ and $\bC \delta_0^n + L(\bR^n)$ with the set of all $\nu\in M(\bR^n)$ that are absolutely continuous with respect to $\delta_0^n + \lambda^n$. Then both $L(\bR^n)$ and $\bC \delta_0^n + L(\bR^n)$ are closed subspaces of $M(\bR^n)$ (e.g. Cohn~\cite[Exercise 4.2.7]{Cohn2013}) and the convolution of two elements in $L(\bR^n)$ is again in $L(\bR^n)$, and the same is true for $\bC \delta_0^n + L(\bR^n)$. Hence both $L(\bR^n)$ and $\bC \delta_0^n + L(\bR^n)$ are closed subalgebras, hence themselves commutative Banach algebras. While $\bC \delta_0^n + L(\bR^n)$ is unital, $L(\bR^n)$ is not. However, every non-unital commutative Banach algebra $A$ can be embedded into a commutative unital Banach algebra, the \emph{unitisation} of $A$ (see e.g. Allan~\cite[Sect.~4.3]{Allan} or Kaniuth~\cite[comment after Ex. 1.1.7]{Kaniuth}), and it is easily seen that $\bC \delta_0^n + L(\bR^n)$ is the unitisation of $L(\bR^n)$.\\
(d) Let $G$ be a locally compact abelian group with Haar measure $H$, e.g.~$\bR^n$ or $\bR^n_d$. Then $L^1(G)$   forms a commutative  Banach algebra with the usual $L^1$-norm and the convolution on $L^1(G)$ as product (e.g.~\cite[Sect.~1.3]{Kaniuth}). When $G=\bR^n$, $L^1(\bR^n)$ can be identified with $L(\bR^n)$, and when $G=\bR^n_d$, then $L^1(\bR^n_d)$ can be identified with the set of all discrete complex measures in $M(\bR^n)$.\\
(e) We now generalise (d) when $G=\bR^n_d$ to the case where the functions may take values in a Banach algebra. So let $A$ be a commutative unital Banach algebra. Then $L^1(\bR^n_d,A)$ consists of all functions $f:\bR^n \to A$  for which there exists a countable set $B\subset \bR^n$ such that $f(z) = 0$ whenever $z\in \bR^n \setminus B$ and $\sum_{z\in B} \|f(z)\|_A <\infty$. With the aid of the indicator function, $f$ can be written as $f=\sum_{x\in B} a_x \one_{\{x\}}$, where $(a_x)_{x\in B}$ is an absolutely summable sequence in $A$ (we have $a_x = f(x)$ for $x\in B$). The convolution for $f,g \in L^1(\bR^n_d,A)$ with $f= \sum_{x\in B} a_x \one_{\{x\}}$ and $g= \sum_{x\in C} c_x \one_{\{x\}}$ is given by
$f\ast g (z) = \sum_{x\in C} f(z-x) g(x) = \sum_{x\in B} g(z-x) f(x)$, $z \in \bR^n$,
and the $L^1$-norm of $f= \sum_{x\in B} a_x \one_{\{x\}} \in L^1(\bR^n_d,A)$ is defined by
$$\|f\|_{L^1(\bR^n_d,A)}  :=
\sum_{x\in B} \|a_x\|_A = \sum_{x\in B} \|f(x)\|_A.
$$
It is easily checked that with convolution and this norm, $L^1(\bR^n_d,A)$ becomes a commutative unital Banach algebra, where the identity $e_{L1d}$ is given by $e_{L1d} = e_A \one_{\{0\}}$, see also~\cite[p.~32]{Kaniuth}. The Banach algebra $L^1(\bR^n_d,A)$ can be identified with the projective tensor product $L^1(\bR^n_d) \widehat{\otimes}_\pi A$, see~\cite[Prop.~1.5.4]{Kaniuth}.\\
(f) Let $A$ be a commutative unital Banach algebra and denote by $\GFSA$ (standing for \lq generalised Fourier series\rq) the set of all $A$-valued functions on $\bR^n$ of the form $z\mapsto \sum_{x\in B} \re^{\ri x^T z} a_x$ for some countable set $B\subset \bR^n$ and absolutely summable elements $(a_x)_{x\in B}$ in $A$. Endowed with the usual addition and multiplication by scalars, and with the usual multiplication of functions, this becomes an algebra with identity $e_{GFS}$, where $e_{GFS}(z) = e_A$ for all $z\in \bR^n$. Consider the map
$$\Psi_A : L^1(\bR^n_d,A) \to \GFSA, \quad \sum_{x\in B} \one_{\{x\}} a_x \mapsto \left[ \bR^n \ni z \mapsto \sum_{x\in B} \re^{\ri x^T z} a_x \right].$$
It is easily seen that $\Psi_A$ is linear, surjective and maps convolution to products. Further, $\Psi_A$ is injective, for if $\sum_{x\in B} \re^{\ri x^T z} a_x = 0_A$ for all $z\in \bR^n$, then
$\sum_{x\in B} \re^{\ri x^T z} \varphi(a_x) =0$ for all $z\in \bR^n$ and all $\varphi$ in the Banach space dual $A^*$ of $A$. Since $z\mapsto \sum_{x\in B} \re^{\ri x^T z} \varphi(a_x)$ is the characteristic function of the complex measure $\sum_{x\in B} \varphi(a_x) \delta_x^n$, we conclude that $\varphi(a_x) = 0$ for all $x\in B$ and $\varphi \in A^*$, hence $a_x = 0$ for all $x\in B$ as a consequence of the Hahn--Banach theorem,
thus proving injectivity of $\Psi_A$. This shows that $\Psi_A$ is an algebra isomorphism from $L^1(\bR^n_d,A)$ onto $\GFSA$. Endowing $\GFSA$ with the norm
$$\|f\|_{\GFSA} := \|\Psi_A^{-1}(f)\|_{L^1(\bR^n_d,A)},$$
i.e.~$\|\sum_{x\in B} \re^{\ri x^T \cdot} a_x\|_{\GFSA} = \sum_{x\in B} \|a_x\|_A$,
it is clear that $\GFSA$ becomes a Banach algebra, isometrically isomorphic to $L^1(\bR^n_d,A)$. Invertibility of elements in the Banach algebra $\GFSA$ will be our main concern in Theorems~\ref{t-characterisation2} and~\ref{t-characterisation}.
\end{example}

Now let $A$ be a commutative Banach algebra, not necessarily unital. Then the \emph{Gelfand space} $\Delta(A)$ consists of all non-zero linear maps $\varphi:A \to \bC$ that are additionally multiplicative, i.e.~which satisfy $\varphi(ab) = \varphi(a) \varphi(b)$ for all $a,b\in A$. Every $\varphi \in \Delta(A)$ is bounded with operator norm $\leq 1$, see e.g.~\cite[Thm.~4.43]{Allan} or \cite[Lem.~2.1.5]{Kaniuth}).  In particular, $\Delta(A)$ is a subset of the Banach space dual $A^*$ of $A$. The space $\Delta(A)$ is equipped with the weak-$*$ topology, which is the weakest topology on $\Delta(A)$ making the maps $\Delta(A) \to \bC$, $\varphi \mapsto \varphi(a)$, for each $a\in A$ continuous. It is easily seen to coincide with the relative topology on $\Delta(A)$ of the weak-$*$-topology on $A^*$. The space $\Delta(A)$ is then a locally compact Hausdorff space (e.g. \cite[Thm.~2.2.3]{Kaniuth}). A commutative Banach algebra $A$ is called \emph{semisimple} if $\Delta(A)$ separates points, equivalently if for every $a\in A \setminus \{0_A\}$ there exists some $\varphi \in \Delta(A)$ such that $\varphi(a) \neq 0$.

When the commutative Banach algebra $A$ is also unital, then the Gelfand space $\Delta(A)$ is non-empty, compact, and each $\varphi\in \Delta(A)$ has operator norm 1 and satisfies
$\varphi(e_A) = 1$ and consequently $\varphi(a^{-1}) = 1/\varphi(a)$ for $a\in A^{-1}$, so that $\varphi(a) \neq 0$ for all $a\in A^{-1}$. Even more, an element $a\in A$ is invertible if and only if $\varphi(a) \neq 0$ for all $\varphi \in \Delta(A)$, and  the spectrum $\mbox{Sp}_A(a)$ of an element $a\in A$ is described by $\{\varphi(a) : a \in \Delta (A)\}$. All this can be found in \cite[Thm.~4.43, Cor.~4.47, Thm.~4.54]{Allan} or \cite[Thm.~2.1.2, Lem.~2.1.5, Thm.~2.2.3, Thm.~2.2.5]{Kaniuth}.

In the following, we will be particularly interested in the Banach algebras $L^1(\bR^n)$, $L(\bR^n)$, $\bC \delta_0^n + L(\bR^n)$, $L^1(\bR^n_d,A)$ and $\GFSA$.

\begin{example} \label{ex-Banach-algebra-2}
(a) For the study of $L^1(\bR^n)$, let more generally $G$ be a locally compact abelian group, e.g. $\bR^n$ or $\bR^n_d$.  It is known that $L^1(G)$ is semisimple, see \cite[Cor.~4.35]{Folland} or \cite[Cor.~2.7.9]{Kaniuth}. For the description of $\Delta (L^1(G))$, recall that a \emph{character} on $G$ is a continuous mapping $\beta:G \to \{ w \in \bC: |w| = 1\}$ such that $\beta(g+h) = \beta(g) \beta(h)$ for all $g,h\in G$. The set of all characters is called the \emph{dual group of $G$} and denoted by $\widehat{G}$. For each $\beta \in \widehat{G}$ define $\varphi_\beta: L^1(G) \to \bC$ by $\varphi_\beta(f) := \int_G f(x) \overline{\beta(x)} \, H(\di x)$, $f\in L^1(G)$ (where $H$ is the Haar measure on $G$). Then $\varphi_\beta \in \Delta(L^1(G))$ and the mapping $\beta \to \varphi_\beta$ is a bijection from $\widehat{G}$ onto $\Delta (L^1(G))$, so that $\Delta (L^1(G))$ can be identified with $\widehat{G}$, see \cite[Thm.~2.7.2]{Kaniuth}. The topology on $\widehat{G}$, induced by this bijection and the weak-$*$ topology on $\Delta(L^1(G))$, coincides with the compact-open topology on $\widehat{G}$, see \cite[Thm.~2.7.5]{Kaniuth}. When $G=\bR^n$, then the dual group (and hence the Gelfand space $\Delta(L^1(\bR^n))$ can be identified with $\bR^n$ with the euclidian topology, see \cite[Ex.~2.7.6~(1,4)]{Kaniuth}. In particular, the Gelfand space of $L^1(\bR^n)$ (and hence of $L(\bR^n)$) is connected, and $L^1(\bR^n)$ is semisimple, but not unital.\\
(b) Since every character on $\bR^n$ is also a character on $\bR^n_d$, we have $\widehat{\bR^n_d} \supset \widehat{\bR^n}$ and hence $\Delta(L^1(\bR^n_d)) \supset \Delta (L^1(\bR^n))$.
Since $\widehat{\bR^n} =\bR^n$, it follows from the Bohr compactification (e.g.~Rudin~\cite[Thm.~1.8.2]{Rudin1962}), interchanging the roles of the group and the dual group, that $\widehat{\bR^n}$ can be continuously embedded as a dense subgroup of $\widehat{\bR^n_d}$. Since $\widehat{\bR^n}  =\bR^n$ is connected, so is any continuous image of it, hence also the closure of such a continuous image. It follows that $\widehat{\bR^n_d}$ and hence $\Delta(L^1(\bR^n))$ are connected, so that $L^1(\bR^n_d)$ is a semisimple commutative unital Banach algebra with connected Gelfand space.\\
(c) As seen in Example~\ref{ex-Banach-algebra-1}~(c), $\bC \delta_0^n + L(\bR^n)$ is the unitisation of $L(\bR^n)$, and $L(\bR^n)$ can be identified with $L^1(\bR^n)$, whose Gelfand space is identified with $\bR^n$. But the Gelfand space of a unitisation is the one-point compactification of the underlying Gelfand space (see \cite[Thm.~2.2.3]{Kaniuth}), hence $\Delta ( \bC \delta_0^n + L(\bR^n))$ can be identified with the one-point compactification of $\bR^n$, in particular it is connected, and $\bC \delta_0^n + L(\bR^n)$ is semisimple (the latter follows from \cite[Rem.~2.1.3]{Kaniuth}).\\
(d) Let $A$ be a commutative unital Banach algebra and consider $L^1(\bR^n_d,A)$ as introduced in Example~\ref{ex-Banach-algebra-1}~(e). Then the Gelfand space of $L^1(\bR^n_d,A)$ can be identified  with the topological product space $\Delta (L^1(\bR^n_d)) \times \Delta (A)$, see \cite[Prop.~1.5.4, Lem.~2.11.1, Thm.~2.11.2]{Kaniuth} (we will take this up in more detail again in the proof of Theorem~\ref{t-characterisation2}). In particular, if $\Delta(A)$ is connected, then so is $\Delta(L^1(\bR^n_d,A))$, since the topological product of connected spaces is connected. If $A$ is semisimple, then so is $L^1(\bR^n_d, A)$, see \cite[Thm.~2.11.8]{Kaniuth}. Since $\GFSA$, as introduced in Example~\ref{ex-Banach-algebra-1}~(f), is isometrically isomorphic to $L^1(\bR^n_d,A)$, similar statements hold for $\GFSA$.\\
(e) The Gelfand space of $M(\bR^n)$ is quite complicated, see Taylor~\cite[Rem.~1.3.5]{T73}.
We shall not go into further detail. We remark however that $M(\bR^n)$ is semisimple, see e.g. Folland~\cite[Cor.~4.35]{Folland}.
\end{example}

\subsection{The distinguished logarithm of Banach algebra valued functions}
Let $A$ be a commutative unital Banach algebra and
suppose that $f:\bR^n \to A$ is continuous such that $f(z) \in A^{-1}$ for all $z\in \bR^n$. Is it then possible to choose a \emph{continuous} logarithm $g:\bR^n \to A$ such that $f(z) = f(0) \exp(g(z))$ for all $z\in \bR^n$, and if so, is it unique? Such a result is clearly true when $A=\bC$ when $g$ is the distinguished logarithm of $f/f(0)$ (see the end of Section~\ref{S-prelim}) and we shall now show that is still true when the Banach algebra $A$ is  semisimple with connected Gelfand space.
We need the following lemma.

\begin{lemma} \label{l-uniqueness}
Let $A$ be a semisimple commutative unital Banach algebra with connected Gelfand space $\Delta(A)$ and identity denoted by $e_A$. Let $a,b \in A$ such that $\exp(a) = \exp(b)$. Then
\begin{equation}  \label{eq-tilde2}
a- \varphi (a) e_A = b - \varphi(b) e_A \quad \forall\; \varphi \in \Delta(A)
\end{equation}
and
there is a constant $k\in \bZ$, which depends only on $a$ and $b$, such that
\begin{equation} \label{eq-tilde4}
\varphi(a) - \varphi(b) = 2\pi \ri k \quad \forall\; \varphi \in \Delta(A).
\end{equation}
\end{lemma}

\begin{proof}
Since each $\varphi\in \Delta(A)$ is continuous, linear and multiplicative, we have
\begin{equation} \label{eq-exponential-phi}
\varphi (\exp (a)) = \varphi \left( \sum_{k=0}^\infty \frac{a^k}{k!} \right) = \sum_{k=0}^\infty \frac{\varphi (a)^k}{k!} = \exp (\varphi(a)),
\end{equation}
and similarly $\varphi(\exp(b)) = \exp(\varphi(b))$. Since, by assumption, $\varphi(\exp(a)) = \varphi(\exp(b))$, we conclude that for each $\varphi \in \Delta (A)$
there exists some $k_{\varphi}\in \bZ$ such that
$
\varphi(a)=\varphi(b)+2\pi \ri k_{\varphi}$.
Since $\Delta(A)$ is equipped with the weak-$\ast$-topology which makes all maps $\Delta(A) \to \bC$, $\varphi \mapsto \varphi(c)$, for $c\in A$ continuous, and since $\Delta(A)$ is connected by assumption, the set $\{\varphi(a-b): \varphi \in \Delta(A)\}$ is connected as well, so we conclude that $k_{\varphi}$ does not depend on~$\varphi$, giving \eqref{eq-tilde4}.
Now let $\varphi, \varphi' \in \Delta(A)$. Since $\varphi(a) - \varphi(b) = 2\pi \ri k$ we have
$\varphi'(\varphi(a) e_A - \varphi(b) e_A ) = 2 \pi \ri k$, and since also $\varphi'(a) - \varphi'(b) = 2 \pi \ri k$ by \eqref{eq-tilde4}, we conclude
$$\varphi' ( a - \varphi(A) e_A - b + \varphi(b) e_A) =  0 \quad \forall\; \varphi, \varphi' \in \Delta(A).$$
Fixing $\varphi$ and letting $\varphi'$ run through $\Delta(A)$ this implies \eqref{eq-tilde2} since $A$ is semisimple.
\end{proof}

\begin{theorem}[Distinguished logarithm for Banach algebra valued functions]\label{t-distinguished}
Let $A$ be a semisimple commutative unital Banach algebra  such that the Gelfand space $\Delta(A)$ is connected. Let $f:\bR^n \to A$ be a continuous function such that $f(z) \in A^{-1}$ for all $z\in \bR^n$. Then there is a unique continuous function $g:\bR^n \to A$ such that $g(0) = 0_A$ and
$$
f(z)=f(0)\exp(g(z)) \quad \forall\; z \in \bR^n.
$$
\end{theorem}

In analogy with the $\bC$-valued case, we shall call the function $g$ appearing in Theorem~\ref{t-distinguished} the \emph{distinguished logarithm of $f(0)^{-1} f$}.

\begin{proof}
Let us first show the existence of the function $g$.  Since $\bR^n$ is connected and since the function $\bR^n \to A^{-1}$, $z\mapsto f(0)^{-1} f(z)$, is continuous, its image is connected as well, so all elements $f(0)^{-1} f(z)$ lie in the connected component of $A^{-1}$ that contains the identity $e_A$, and as mentioned in Section~\ref{S-Banach-general}, that component is known to be $\exp(A) = \{ \exp(a) : a \in A\}$. Hence, for each $z\in \bR^n$, there exists some $g_1(z) \in A$ such that
$f(z) = f(0) \exp (g_1(z))$. Fix $\varphi_0 \in \Delta(A)$. We claim that
\begin{equation} \label{eq-tilde3}
\bR^n \to A, \quad z \mapsto g_1(z) - \varphi_0(g_1(z)) e_A \quad \mbox{is continuous.}
\end{equation}
To see this, let $z_0 \in \bR^n$. By continuity of $f$, there is an open neighbourhood $U_{z_0}$ of $z_0$ such that
$\|e_A - f(z_0)^{-1} f(z)\|_A < 1$ for all $z\in U_{z_0}$. However, on $\{ a \in A: \|e_A - a \|_A < 1\}$ a logarithm can be defined by the logarithmic series $a \mapsto \log a := -\sum_{k=1}^\infty k^{-1} (e_A - a)^k$ so that $\exp (\log(a)) = a$  (see e.g.~Allan~\cite[Thm. 4.100, Prop. 4.101]{Allan}), and it is easily seen that this map is continuous. Denoting then $g_2(z) := \log (f(0)^{-1} f(z))$ for $z \in U_{z_0}$, we see that $g_2$ is continuous on $U_{z_0}$ with $\exp (g_2(z)) = f(0)^{-1} f(z) = \exp (g_1(z))$. Then also $z\mapsto \varphi_0 (g_2(z)) e_A$ is continuous on $U_{z_0}$ and we conclude that $z\mapsto g_2(z) - \varphi_0 (g_2(z)) e_A$ is continuous on $U_{z_0}$. By \eqref{eq-tilde2}, the same is true for $z\mapsto g_1(z) - \varphi_0 (g_1(z)) e_A$. Since $z_0 \in \bR^n$ was arbitrary, this gives \eqref{eq-tilde3}.

Now define
$$h: \bR^n \to \bC, \quad z \mapsto \varphi_0 (f(0)^{-1} f(z)).$$
This is a continuous function with $h(0) = \varphi_0(e_A) = 1$ and $h(z) \neq 0$ for all $z\in \bR^n$. Denote by $\psi:\bR^n \to \bC$ the distinguished logarithm of $h$ and define $$g(z) := g_1(z) - \varphi_0 (g_1(z)) e_A + \psi(z) e_A \quad \forall\; z \in \bR^n.$$
Then $g$ is continuous by \eqref{eq-tilde3}, and for $z\in \bR^n$ we obtain with a similar calculation as in \eqref{eq-exponential-phi},
\begin{align*}
\exp(g(z)) & = \exp (g_1(z)) \, \varphi_0 (\exp (-g_1(z))) e_A \; h(z) e_A \\
& = \exp(g_1(z)) \, \varphi_0 (f(z)^{-1} f(0)) e_A \; \varphi_0(f(0)^{-1} f(z)) e_A \\
& = \exp (g_1(z)) = f(0)^{-1} f(z).
\end{align*}
Finally, $g(0) = g_1(0) - \varphi_0(g_1(0)) e_A$ which by \eqref{eq-tilde2} is equal to $g_2(0) - \varphi_0(g_2(0)) e_A$ with $g_2$ constructed on $U_0$ as before with $z_0=0$.  But $g_2(0) = 0_A$, hence $g(0) = 0_A$, finishing the proof of the existence of $g$.

To show uniqueness of $g$, let $g_3:\bR^n \to \bC$ be also a continuous function with $g_3(0) = 0_A$ and $f(z) = f(0) \exp (g_3(z))$ for all $z\in \bR^n$. From \eqref{eq-tilde2} and \eqref{eq-tilde4} we conclude
\begin{equation*}
g(z) - g_3(z) = \varphi_0 (g(z) - g_3(z)) e_A \in (2\pi \ri  \bZ)e_A.
\end{equation*}
Since $g-g_3$ is continuous, it must be constant, and since $g(0) = g_3(0)$ we conclude that $g_3(z) = g(z)$ for all $z\in \bR^n$, thus establishing uniqueness.
\end{proof}

\subsection{Almost periodic $A$-valued functions}

In Section~\ref{S7-5} we will aim at a characterisation of invertible elements of $\GFSA$. Elements of $\GFSA$ are in particular almost periodic, and as an auxiliary result, we shall now show that the distinguished logarithm of suitable $A$-valued almost periodic functions, when \lq suitably corrected\rq, is again almost periodic. We start with some definitions.

Let $A$ be a commutative unital Banach algebra with norm $\|\cdot\|_A$.
For a function $f:\bR^n \to A$ denote by
$$\| f \|_{\infty,A} := \sup_{z\in \bR^n} \|f(z)\|_A$$
the supremum norm of the $A$-valued function, and let $C_b(\bR^n,A)$ be the set of all bounded continuous $A$-valued functions $f$ on $\bR^n$, i.e.~which satisfy $\|f\|_{\infty,A} < \infty$. Endowed with the norm $\|\cdot\|_{\infty,A}$, it is easily checked that $C_b(\bR^n,A)$
becomes a unital commutative Banach algebra, with identity $e_{CbA}$ given by
$e_{CbA}(z) = e_A$ for all $z\in \bR^n$. As in Pankov~\cite[p.~7]{Pankov}, we call a function $f:\bR^d \to A$ \emph{almost periodic (in the sense of Bohr)}, if it is in $C_b(\bR^n,A)$ and for each $\varepsilon > 0$ there is a compact set $K_\varepsilon \subset \bR^n$ such that
each set of the form $\alpha + K_\varepsilon$ with $\alpha \in \bR^n$ contains at least one element $\tau_{\varepsilon,\alpha}$ satisfying $\|f(\tau_{\varepsilon,\alpha} + \cdot) - f\|_{\infty,A} < \varepsilon$. The set of all $A$-valued periodic functions on $\bR^n$ will be denoted by $\CAP$.

\begin{remark} \label{AP-equivalences}
(a) $\CAP$ is a Banach space with norm $\|\cdot\|_{\infty,A}$, see \cite[p.9]{Pankov}.\\
(b) Let $f\in C_b(\bR^n,A)$. Then the following are equivalent.
\begin{enumerate}
\item[(i)] $f$ is almost periodic (in the sense of Bohr), i.e.~$f\in \CAP$.
\item[(ii)] $f$ can be represented as a uniform limit of trigonometric polynomials, i.e. there exists a sequence $(f_k)_{k\in \bN}$ in $C_b(\bR^n,A)$, where each $f_k$ is of the form $f_k = \sum_{x\in B_k} \re^{\ri x^T \cdot} a_{k,x}$ for some finite set $B_k\subset \bR^n$ and $a_{k,x} \in A$, such that $\lim_{k\to\infty} \|f_k - f\|_{\infty,A} = 0$.
\item[(iii)] $f$ is almost periodic in the sense of Bochner, i.e.~for every sequence $(z_k)_{k\in \bN}$ the sequence $(f(z_k + \cdot))_{k\in \bN}$ is relatively compact in $C_b(\bR^n,A)$, i.e.~has a convergent subsequence.
\end{enumerate}
See Pankov~\cite[Thm. 1.2]{Pankov} for the equivalence of (i) and (iii). By  \cite[Prop.~1.3]{Pankov}, the space of trigonometric polynomials is dense in $\CAP$ with respect to convergence in $\|\cdot\|_{\infty,A}$, giving (together with (a)) the equivalence of (i) and (ii). Even more, $\CAP$ is the closure of the set of trigonometric polynomials with respect to $\|\cdot\|_{\infty,A}$ in $C_b(\bR^n,A)$.\\
(c) It follows easily from (ii) in (b) that $\CAP$ is closed under multiplication of functions and that $e_{CbA} \in \CAP$, hence $\CAP$ is a closed unital commutative Banach sub-algebra of $C_b(\bR^n,A)$ with identity $e_{CbA}$.\\
(d) Let $f\in \GFSA$ with $f(z) = \sum_{x\in B} \re^{\ri x^T z} a_x$, $B\subset \bR^n$ being countable and $(a_x)_{x\in B}$ being absolutely summable in $A$. Then $f$ is clearly in $C_b(\bR^n,A)$ and a uniform limit of trigonometric polynomials, hence is almost periodic. It follows that $\GFSA \subset \CAP$ and that  $\CAP$ is the closure of $\GFSA$ \emph{with respect to the norm $\|\cdot\|_{A,\infty}$}. Observe however that $\GFSA$, as a Banach algebra,  carries a norm different from $\|\cdot\|_{\infty,A}$, namely the norm introduced in Example~\ref{ex-Banach-algebra-1}~(f). In general (for example when $A=\bC$), we have $\GFSA \subsetneq \CAP$.
\end{remark}

We can now formulate and prove the aforementioned result regarding the \lq suitably corrected\rq~distinguished logarithm.

\begin{lemma} \label{lem-ap}
Let $A$ be a semisimple commutative unital Banach algebra with connected Gelfand space $\Delta(A)$. Let $f\in \CAP$ such that $f(z) \in A^{-1}$ for all $z\in \bR^n$ and assume that $\sup_{z\in \bR^n} \|f(z)^{-1}\|_A < \infty$. Denote by $g$ the distinguished logarithm of $f(0)^{-1} f$, as defined in Theorem~\ref{t-distinguished}. For fixed $\varphi_0 \in \Delta(A)$, define $G : \bR^n \to A$ by $G(z) := g(z) - \varphi_0(g(z)) e_A$. Then $G \in \CAP$.
\end{lemma}

\begin{proof}
Replacing $f$ by $f(0)^{-1} f$ we can assume without loss of generality that $f(0) = e_A$. Then $f(z) = \exp (g(z))$ for all $z\in \bR^n$.
Fix $\varphi_0 \in \Delta(A)$ and let $G$ be defined as in the statement of the lemma.
Since $f\in \CAP$, by the definition of almost periodicity we conclude that for every $\varepsilon > 0$  there exists a compact set $K_\varepsilon\subset \mathbb{R}^n$ such that
\begin{align} \label{K-epsilon}
\sup\limits_{x\in \mathbb{R}^n} \inf_{y\in K_\varepsilon}\|{f}(y)-{f}(x)\|_A <\varepsilon/2.
\end{align}
Since $g$ is continuous, so is $G$. Our first goal is to prove that $G$ is bounded, the second to show that $G$ is almost periodic. To see that $G\in C_b(\bR^n,A)$, define $\varepsilon_1 :=  (\sup_{z \in \bR^n} \| f(z)^{-1}\|_A)^{-1}>0$. By \eqref{K-epsilon}, for every  $z\in \mathbb{R}^n$  there exists some $y_z\in K_{\varepsilon_1}$ such that $$\|f(y_z)-f(z)\|_A\leq \frac{1}{2\sup_{x\in \mathbb{R}^n}\|{f}(x)^{-1}\|_A}.$$ Writing
$${f}(z)={f}(y_z)(e_A+ {f}(y_z)^{-1}({f}(z)-{f}(y_z)))$$ and observing that $\|{f}(y_z)^{-1}({f}(z)-{f}(y_z))\|_A \leq 1/2$, as in the proof of Theorem~\ref{t-distinguished} we can apply the logarithmic series $\log(e_A+a) := \sum_{j=1}^\infty (-1)^{j+1} j^{-1} a^j$ when $\|a\|_A < 1$, so that $f(z) = \exp [h(z)]$ with
$$
{h}(z):= g(y_z)  + \log (e_A+ {f}(y_z)^{-1}({f}(z)-{f}(y_z))).$$
Since $a\mapsto \log(e_A + a)$ is obviously bounded on $\{a\in A : \|a\|_A\leq 1/2\}$, since $y_z$ is in the compact set $K_{\varepsilon_1}$ for every $z\in \R^n$, and since $g$ is continuous, this implies boundedness of $h$. Hence also $\bR^n\to A$, $z \mapsto h(z) - \varphi_0(h(z)) e_A$, is bounded. By \eqref{eq-tilde2}, the latter is equal to $G(z)$, showing boundedness of $G$ so that $G\in C_b(\bR^n,A)$.

To show that $G$ is almost periodic, we use the characterisation (iii) in Remark~\ref{AP-equivalences}~(b). So let $(z_k)_{k\in \bN}$ be a given sequence in $\bR^n$.
Since $f$ is almost periodic, there exists a convergent subsequence of $({f}(z_k+\cdot))_{k\in \bN}$ in $C_b(\mathbb{R}^n,A)$, and for simplicity we denote the corresponding subsequence again by $(f(z_k+\cdot))_{k\in \bN}$. Now let $\varepsilon > 0$. Since $(f(z_k+\cdot))_{k\in \bN}$ is a Cauchy sequence, there exists $k_0=k_0(\varepsilon) \in \mathbb{N}$ such that $\| f(z_k + \cdot) - f(z_m + \cdot)\|_{\infty,A} < \varepsilon$ for all $k,m\geq k_0$.
Write
\begin{align*}
 {f}(z_m+z)
=&{f}(z_{k_0}+z)(e_A+{f}(z_{k_0}+z)^{-1}({f}(z_m+z)-{f}(z_{k_0}+z))).
\end{align*}
When $\varepsilon < (2\sup_{x\in \bR^n} \|f(x)^{-1}\|)^{-1}$, we have
$\|{f}(z_{k_0}+z)^{-1}({f}(z_m+z)-{f}(z_{k_0}+z))\|_A < 1/2$ for all $m\geq k_0$ and $z\in \bR^n$, hence we can apply the logarithmic series, and with
\begin{align}
l_m(z) & := \log  \left[ e_A+{f}(z_{k_0}+z)^{-1}({f}(z_m+z)-{f}(z_{k_0}+z))\right] \nonumber \\
& =  \sum_{j=1}^\infty (-1)^{j+1} j^{-1} f(z_{k_0}+z)^{-j} (f(z_m+z) - f(z_{k_0} + z))^j \label{log-series}
\end{align}
we obtain
$$\exp [g(z_m + z)] = f(z_m + z) = \exp [g(z_{k_0} + z)] \, \exp [l_m(z)]$$
for all $m\geq k_0$ and $z\in \bR^n$. An application of \eqref{eq-tilde2} then gives
$$G(z_m+z) = G(z_{k_0} + z) + l_m(z) - \varphi_0(l_m(z)) e_A,$$
hence
\begin{align*}
\|G(z_m+z)-G(z_k+z)\|_A=&\|l_m(z)-\varphi_0(l_m(z))e_A-(l_k(z)-\varphi_0(l_k(z))e_A)\|_A\\
\leq&2\|l_m(z)-l_k(z)\|_A
\end{align*}
for all $k,m\geq k_0$ and $z\in \bR^n$ (observe that the operator norm of $\varphi_0$ is 1).
Using that
$$(a-b)^j - (c-b)^j = (a-c) \sum_{i=0}^{j-1} (a-b)^i (c-b)^{j-i-1}$$
for arbitrary $a,b,c\in A$ and $j\in \bN$, we obtain from the definition~\eqref{log-series} of $l_m(z)$
\begin{align*}
&\|l_m(z)-l_k(z)\|_A\\
\leq& \|{f}(z_m+z)-{f}(z_k+z)\|_A\\
&\cdot\sum\limits_{j=1}^\infty j^{-1} \|{f(z_{k_0}+z)}^{-1}\|_A^{j}\sum\limits_{i=0}^{j-1}\|{f}(z_m+z)-{f}(z_{k_0}+z)\|_A^i
\|{f}(z_k+z)-{f}(z_{k_0}+z)\|_A^{j-1-i}\\
\leq&\varepsilon \sum\limits_{j=1}^\infty \|{f}^{-1}\|_{\infty,A}^j \varepsilon^{j-1}
\end{align*}
whenever $z\in \bR^n$, $n,m\geq k_0(\varepsilon)$ and $\varepsilon$ is sufficiently small. Letting $\varepsilon\to 0$ we conclude that   $(G(z_k+\cdot))_{k\in \mathbb{N}}$ is a Cauchy sequence in $C_b(\bR^n,A)$, hence convergent. This establishes that $G$ is almost periodic.
\end{proof}

Observe that Lemma~\ref{lem-ap} does not give any new information when $A=\bC$, because $\Delta(\bC)$ consists only of the identity map on $\bC$, hence $G$ constructed in Lemma~\ref{lem-ap} is the zero-map.

\subsection{Invertibility in $\GFSA$} \label{S7-5}

We now come to the main goal of Section~\ref{S6}, namely the characterisation of invertible elements in the Banach algebra $\GFSA$, equivalently in $L^1(\bR^n_d,A)$. In particular, in Theorem~\ref{t-characterisation} we shall obtain a characterisation in terms of exponentials, similar to \eqref{eq-7-a}, where $A$ is a semisimple commutative unital Banach algebra with connected Gelfand space. Other characterisations, without assuming $A$ to be semisimple or to have connected Gelfand space, will already be obtained in Theorem~\ref{t-characterisation2}.

We know that $\GFSA\subset C_b(\bR^n,A)$ and both share the same identity, however, both Banach algebras carry different norms, namely $\|\cdot\|_{\GFSA}$ and $\|\cdot\|_{\infty,A}$, respectively. It is easily seen that
\begin{equation} \label{eq-diff-norms}
\|f\|_{\infty,A} \leq \|f\|_{\GFSA} \quad \forall\; f \in \GFSA,
\end{equation}
but not necessarily do we have equality. A natural question now is whether $\GFSA$ is \emph{inverse closed} in $C_b(\bR^n,A)$, that is whether every $f\in \GFSA$ that is invertible \emph{in $C_b(\bR^n,A)$} is  automatically invertible \emph{in $\GFSA$}; in other words, do we have $\GFSA \cap (C_b(\bR^n,A))^{-1} = (\GFSA)^{-1}$? This is indeed the case and we will prove it now together with some other characterisations of invertibility in $\GFSA$. Recall the isometric isomorphism $\Psi_A: L^1(\bR^n_d,A) \to \GFSA$ introduced in Example~\ref{ex-Banach-algebra-1}~(f).

\begin{theorem}[First characterisation of invertible elements in $\GFSA$] \label{t-characterisation2}
Let $A$ be a commutative unital Banach algebra and let $f\in \GFSA$. Then the following are equivalent:
\begin{enumerate}
\item[(i)] $f\in (\GFSA)^{-1}$, i.e.~$f$ is invertible in $\GFSA$.
\item[(ii)] $f \in (C_b(\bR^n,A))^{-1}$, i.e.~$f$ is invertible in $C_b(\bR^n,A)$.
\item[(iii)] $f(z) \in A^{-1}$ for all $z\in \bR^n$ and $\sup_{z\in \bR^n} \|f(z)^{-1}\|_A < \infty$.
\item[(iv)] $\Psi_A^{-1}(f) \in (L^1(\bR^n_d,A))^{-1}$, i.e.~$\Psi_A^{-1}(f)$ is invertible in $L^1(\bR^n_d,A)$.
\item[(v)]
The functions $\bR^n \to \bC$, $z\mapsto \varphi(f(z))$, are invertible in $\GFSC$ for all $\varphi \in \Delta(A)$.
\item[(vi)] $\inf_{z\in \bR^n} |\varphi(f(z))| > 0$ for all $\varphi \in \Delta(A)$.
\end{enumerate}
\end{theorem}

\begin{proof}
Write $f(z) = \sum_{x\in B} \re^{\ri x^T z} a_x$ with $B\subset \bR^n$ countable and $(a_x)_{x\in B}$ being an absolutely summable sequence in $A$. The equivalence of (i) and (iv) is clear since~$\Psi_A$ is an isometric Banach algebra isomorphism. The implications \lq\lq $(i) \Longrightarrow (ii) \Longrightarrow (iii)$\rq\rq~are also clear. To see that (iii) implies (vi), let $\varphi\in \Delta(A)$. Then $\varphi(f(z)) = \sum_{x\in B} \re^{\ri x^T z} \varphi(a_x)$ with $(\varphi(a_x))_{x\in B}$ being absolutely summable in $\bC$, hence $z\mapsto \varphi(f(z))$ is in $\GFSC$. Since $f(z)\in A^{-1}$ for all $z\in \bR^n$ we have $\varphi(f(z)) \neq 0$ for all $z\in \bR^n$, and since $\varphi$ has operator norm 1 we conclude
$$\sup_{z\in \bR^n} |\varphi(f(z))|^{-1} = \sup_{z\in \bR^n} |\varphi(f(z)^{-1})| \leq \sup_{z\in \bR^n} \|f(z)^{-1}\|_A < \infty,$$
thus establishing (vi). That (vi) implies (v) follows from the equivalence of (i) and (ii) in Section~\ref{S-motivation}, i.e.~from the reformulation of Theorem~\ref{t-Khartov}.

It remains to show that (v) implies (i), equivalently that (v) implies (iv). Denote $h:= \Psi_A^{-1}(f)$, so that $h=\sum_{x\in B} \one_{\{x\}} a_x \in L^1(\bR^n_d,A)$. As mentioned in Section~\ref{S-Banach-general}, $h$ will be invertible in $L^1(\bR^n_d,A)$ if we can show that $\rho(h) \neq 0$ for all $h\in \Delta (L^1(\bR^n_d,A))$.
Now let $\rho \in \Delta (L^1(\bR^n_d,A))$ be arbitrary. Then $\rho(h) = \sum_{x\in B} \rho( \one_{\{x\}}a_x)$. By  Kaniuth~\cite[Prop.~1.5.4]{Kaniuth} we can identify
$L^1(\bR^n_d,A)$ with the projective tensor product $L^1(\bR^n_d) \widehat{\otimes}_\pi A$, where
$\one_{\{x\}} a_x$ is identified with the elementary tensor $\one_{\{x\}} \otimes a_x$. Further, by \cite[Lem.~2.11.1, Thm.~2.11.2]{Kaniuth}, $\Delta (L^1(\bR^n_d) \widehat{\otimes}_\pi A)$ can be identified with $\Delta (L^1(\bR^n_d)) \times \Delta(A)$, and to every $\rho' \in
\Delta (L^1(\bR^n_d) \widehat{\otimes}_\pi A)$ there exist $\xi \in \Delta(L^1(\bR^n_d))$ and $\varphi \in \Delta (A)$ such that $\rho'(g \otimes b) = \xi(g) \varphi(b)$ for all elementary tensors $g\otimes b$ with $g\in L^1(\bR^n_d)$ and $b\in A$ (and vice versa, $\xi$ and $\varphi$ determine some $\rho'$). Hence, identifying $\rho\in \Delta (L^1(\bR^n_d,A))$ with $\rho'\in \Delta (L^1(\bR^n_d) \widehat{\otimes}_\pi A)$ we find
$\rho(\one_{\{x\}} a_x) = \xi(\one_{\{x\}}) \varphi(a_x)$, so that
$$\rho(h) = \sum_{x\in B} \xi(\one_{\{x\}}) \varphi(a_x) = \xi \left( \sum_{x\in B} \one_{\{x\}}
\varphi(a_x)\right).$$ But
$$\Psi_\bC \left( \sum_{x\in B} \one_{\{x\}} \varphi(a_x)\right) (z)  = \sum_{x\in B} \re^{\ri x^T z} \varphi(a_x) = \varphi(f(z)),$$
and since $z\mapsto \varphi(f(z))$ is in $(\GFSC)^{-1}$, we have $\sum_{x\in B} \one_{\{x\}} \varphi(a_x)\in (L^1(\bR^n_d, \bC))^{-1}$, hence $\xi(\sum_{x\in B} \one_{\{x\}} \varphi(a_x))\neq 0$ for all $\xi\in \Delta(L^1(\bR^n_d))$ and $\varphi\in \Delta(A)$, giving invertibility of $h$. This completes the proof.
\end{proof}

Since $\GFSA$ is inverse closed in $C_b(\bR^n,A)$, the spectrum of $f\in \GFSA$ does not change when considered as an element in $C_b (\bR^n,A)$. In particular, for $f\in \GFSA$ we have
\begin{equation} \label{eq-spec-estimate}
\mbox{Sp}_{\GFSA}(f) = \mbox{Sp}_{C_b(\bR^n,A)}(f)  \subset \{ w \in \bC: |w| \leq \|f\|_{\infty,A}\},
\end{equation}
where the last inclusion follows from \cite[Thm.~1.2.8]{Kaniuth}. Observe that here we could use the norm $\|\cdot\|_{\infty,A}$ rather than the norm $\|\cdot\|_{\GFSA}$.

Next we establish a counterpart to Lemma~\ref{lem-ap} when $f\in \GFSA$.

\begin{lemma} \label{lem-ap2}
Let $A$, $\varphi_0$, $f$ and $G$ be as in Lemma~\ref{lem-ap}, in particular, $A$ is semisimple with connected Gelfand space. In addition to the assumptions of Lemma~\ref{lem-ap}, assume that $f\in \GFSA$. Then also $G\in \GFSA$.
\end{lemma}

\begin{proof}
In this proof we shall write $\exp$ for the exponential series in $A$ or in $\bC$, $\exp_{CbA}$ for the exponential series in $C_b(\bR^n,A)$ (which has identity $e_{CbA}$), and $\exp_{GFS}$ for the exponential series in $\GFSA$ (which has identity $e_{GFS} = e_{CbA}$).
By \eqref{eq-diff-norms}, convergence  in $\|\cdot\|_{\GFSA}$ implies convergence in $\|\cdot\|_{\infty,A}$, hence we have $\exp_{CbA}(h) = \exp_{GFS}(h)$ for $h\in \GFSA$, and it is easily seen that $(\exp_{CbA}(h))(z) = \exp (h(z))$ for $h\in C_b(\bR^n,A)$. Even though the forms of the exponential series all lead to the same result, we decided to keep the subscripts in this proof for clarity.

Let $\varphi_0,f$ and $G$ be as in the statement of the lemma.
As in the proof of Lem\-ma~\ref{lem-ap}, we can assume without loss of generality that $f(0) = e_A$.
Since $\GFSA \subset \CAP$, we know from Lemma~\ref{lem-ap} that $G\in \CAP$. We need to show that even $G\in \GFSA$. From the implication \lq\lq $(iii) \Longrightarrow (v)$\rq\rq~in Theorem~\ref{t-characterisation2} we know that $z\mapsto \varphi_0(f(z))$ is in $(\GFSC)^{-1}$, hence
$z\mapsto \frac{1}{\varphi_0(f(z))} e_A$ is in $\GFSA$. Since
\begin{align}
\exp(G(z)) & = \exp [g(z) - \varphi_0(g(z)) e_A] = \exp[g(z)] \, \varphi_0 (\exp(-g(z))) e_A \nonumber \\
& = f(z) (\varphi_0(f(z)))^{-1} e_A \label{eq-e-varphi}
\end{align}
(where we used \eqref{eq-exponential-phi}), and since $f\in \GFSA$, we conclude that
the function  $z\mapsto \exp(G(z))$ is in $\GFSA$, hence $\exp_{CbA}(G) \in\GFSA$.
Since $G$ is almost periodic, by Remark~\ref{AP-equivalences}~(b) there exists a sequence $(p_k)_{k\in \bN}$ of trigonometric polynomials such that $\|p_k - G\|_{\infty,A} \to 0$ as $k\to\infty$. Write $G = G - p_k + p_k$.
Since $p_k$ is in the Banach algebra $\GFSA$, also $\exp_{CbA}(-p_k) = \exp_{GFS}(-p_k) \in \GFSA$. Since also $\exp_{CbA}(G) \in \GFSA$, since $\exp_{CbA} (G-p_k) - e_{CbA}
= \exp_{CbA}(G) \exp_{CbA}(-p_k) - e_{CbA}$ and since $e_{CbA} = e_{GFS} \in \GFSA$, we conclude that  $\exp_{CbA} (G-p_k) - e_{CbA} \in \GFSA$.
Since
\begin{align*}
\left\| \exp_{CbA}(G-p_k) - e_{CbA}\right\|_{\infty,A} & \leq \sum_{j=1}^\infty \frac{\|G-p_k\|^j_{\infty,A}}{j!} = \exp (\|G-p_k\|_{\infty,A}) - 1
\end{align*}
and the latter tends to 0 as $k\to\infty$, we conclude from
\eqref{eq-spec-estimate}
that the spectrum of $\exp_{CbA}(G-p_k) - e_{CbA}$ (as an element of $\GFSA$) is contained in $\{ w \in \bC: |w|\leq 1/2\}$ for large $k$. Hence the spectrum of $\exp_{CbA}(G-p_k)$ as an element of $\GFSA$ is contained in $\{w\in \bC: |w-1| \leq 1/2\}$ for large $k$. Hence we can apply the one-variable holomorphic functional calculus in the Banach algebra $\GFSA$ to the principle branch of the complex logarithm (e.g. Kaniuth~\cite[Lem.~3.2.4]{Kaniuth}), and conclude that there is some $H_k\in \GFSA$ such that $\exp_{GFS}(H_k) = \exp_{CbA}(G-p_k)$ for large~$k$.
Hence also $\exp(H_k(z)) = \exp (G(z) - p_k(z))$ for all $z\in \bR^n$. Since $H_k$, $p_k$ and $G$ are continuous, it follows as in the proof of uniqueness of the distinguished logarithm in Theorem~\ref{t-distinguished} that there are $m_k\in \bZ$ such that $H_k(z) = G(z) - p_k(z) + 2 \pi \ri m_k e_A$ for all $z\in \bR^n$, i.e. $G(z)=p_k(z) + H_k(z) - 2\pi \ri m_k \re^{\ri 0^T z} e_A$. Since $H_k$ and $p_k$ are in $\GFSA$, we conclude that also $G\in \GFSA$.
\end{proof}

We can now come to our second characterisation of invertible elements in $\GFSA$, when $A$ is additionally semisimple with connected Gelfand space.

\begin{theorem}[Second characterisation of invertible elements in $\GFSA$] \label{t-characterisation}
Let $A$ be a semisimple commutative unital  Banach algebra with connected Gelfand space~$\Delta(A)$. Let $f\in \GFSA$. Then the following are equivalent:
\begin{enumerate}
\item[(i)] $f\in (\GFSA)^{-1}$, i.e.~$f$ is invertible in $\GFSA$.
\item[(vii)] $f(0) \in A^{-1}$ and there exist some $\gamma\in \bR^n$ and $g_1\in \GFSA$ such that
$f(z) = f(0) \exp (\ri \gamma^T z \, e_A + g_1(z) )$ for all $z \in \bR^n$.
\item[(viii)] $f(0) \in A^{-1}$ and there exist some $\gamma\in \bR^n$ and $g_2\in \GFSA$ such that
$f(z) = f(0) \exp (\ri \gamma^T z \, e_A + g_2(z)-g_2(0) )$ for all $z \in \bR^n$.
\end{enumerate}
\end{theorem}

\begin{proof}
To see that (i) implies (vii), recall that (i) implies (iii) of Theorem~\ref{t-characterisation2}. Denote by $g$ the distinguished logarithm of $f(0)^{-1} f$, fix $\varphi_0 \in \Delta(A)$ and denote $G(z) := g(z) - \varphi_0 (g(z)) e_A$, $z\in \bR^n$. Then $G\in \GFSA$ by Lemma~\ref{lem-ap2}. From \eqref{eq-e-varphi} (recall that in the proof of Lemma~\ref{lem-ap2} we had assumed that $f(0) = e_A$) we obtain
$$f(0)^{-1} f(z) =   \varphi_0 (f(0)^{-1} f(z)) e_A  \exp(G(z)) \quad \forall\; z \in \bR^n.$$
But $z\mapsto \varphi_0 (f(0)^{-1} f(z))$ is in $(\GFSC)^{-1}$ by property (v) of Theorem~\ref{t-characterisation2}. Hence by~\eqref{eq-7-a} there exist $\gamma\in \bR^n$ and $g_0 \in \GFSC$ such that
$$\varphi_0 (f(0)^{-1} f(z)) = \exp(\ri \gamma^T z + g_0(z) - g_0(0)).$$
Defining $g_1(z) := G(z) + (g_0(z) - g_0(0)) e_A$ we have $g_1 \in \GFSA$ and obtain $f(z) = f(0) \exp (\ri \gamma^T z \, e_A + g_1(z))$, thus proving (vii).

To see that (vii) implies (i), consider the function $f': \bR^n \to A$ defined by
\begin{align*}
f'(z) & := f(0)^{-1} \exp (-\ri \gamma^T z \, e_A) \, \exp(-g_1(z)) .
\end{align*}
Recalling that $\exp (-g_1(z)) = (\exp (-g_1))(z)$ (the latter is the exponential series of $-g_1$ in $\GFSA$, evaluated at $z$, see the beginning of the proof of Lemma~\ref{lem-ap2}), and observing that $z\mapsto \exp (-\ri \gamma^T z \, e_A) = \exp(-\ri \gamma^T z) \, e_A$ is in $\GFSA$, we conclude that $f'\in \GFSA$.
Since $f'(z) f(z) = e_A$ for all $z\in \bR^n$ we have $f' f = e_{GFS}$, so that $f$ is invertible in $\GFSA$. Hence (i) and (vii) are equivalent.

For the equivalence of (vii) and (viii), assume first that (vii) holds true. Then $f(0) = f(0) \exp (g_1(0))$, so that $\exp (g_1(0)) = e_A$. Condition (viii) then follows with $g_2 := g_1$. Conversely, if (viii) is true with some $g_2$, then (vii) follows with $z\mapsto g_1(z) := g_2(z) - g_2(0)$, which is then also in $\GFSA$.
\end{proof}

\section{The case $q\in \{1,\ldots, n-1\}$} \label{S7}

We now return to our problem of characterising the invertible complex measures on $\bR^n$. As seen in Proposition~\ref{p-Taylor3}, it is enough to do this for complex measures of certain forms, which we parameterised by some number $q\in \{0,1,\ldots, n\}$. We have already treated the cases $q=0$ and $q=n$ and can now treat the intermediate case. More precisely,
with the help of Theorem~\ref{t-characterisation}
we are now in a position to characterise when an element of the form $\mu = \alpha \delta_0^n + f \lambda^q \otimes \zeta^{n-q}_C$ as appearing in Proposition~\ref{p-Taylor3}, with $q\in \{1,\ldots, n-1\}$, is invertible. We have the following result:

\begin{theorem} \label{t-allg-q}
Let $n\geq 2$, $q\in \{1,\ldots, n-1\}$ and $\mu \in M(\bR^n)$ be of the form $\mu = \alpha \delta_0^n + f \lambda^q \otimes \zeta_C^{n-q}$, where $\alpha \in \bC$, $C\subset \bR^{n-q}$ is countable and $f:\bR^n \to \bC$ is $\lambda^q \otimes \zeta_C^{n-q}$-integrable and vanishes outside $\bR^q \times C$.\\
(a) Suppose that $q=1$. Then the following are equivalent:
\begin{enumerate}
\item[(i)] $\mu$ is invertible.
\item[(ii)] $\mu \in \CLK$ with $\gamma \in \{0_q\} \times \bR^{n-q}$ in \eqref{eq-CLK} and complex L\'evy type measure $\nu$ of the form $\nu = \nu_0 + \nu_1$, where $\nu_0 \in M(\bR^n)$ (hence is finite) and $\nu_1$ is of the form
    $$\nu_1(\di (x,y)) = \left( m \frac{\re^{-|x|}}{x} \lambda^1 (\di x)\right) \otimes \delta_0^{n-1}(\di y), \quad (x,y) \in \bR \times \bR^{n-1}$$
    for some $m\in \bZ$, i.e.~is the product measure of the (infinite when $m\neq 0$) quasi-L\'evy type measure   $m \frac{\re^{-|x|}}{x} \lambda^1 (\di x)$ and $\delta_0^{n-1}$ (the product lives on $\cB_0^n$).
\item[(iii)] There exist $\gamma \in \{0_q\} \times \bR^{n-q}$, $m\in \bZ$ and $\nu' \in M(\bR^n)$ such that $\mu = \delta_\gamma^n \ast (\sigma^{\ast m} \otimes \delta_0^{n-1}) \ast \exp(\nu')$, where $\sigma\in M(\bR)$ is defined in Theorem~\ref{t-Taylor1}.
\end{enumerate}
(b) Suppose that $q\geq 2$. Then the following are equivalent:
\begin{enumerate}
\item[(i)] $\mu$ is invertible.
\item[(ii)] $\mu \in \CLK$ with  $\gamma \in \{0_q\} \times \bR^{n-q}$ in \eqref{eq-CLK} and finite complex L\'evy type measure~$\nu$.
\item[(iii)] There exist $\gamma \in \{0_q\} \times \bR^{n-q}$ and $\nu' \in M(\bR^n)$ such that $\mu = \delta_\gamma^n \ast \exp(\nu')$.
\end{enumerate}
\end{theorem}

\begin{proof}
We write $(x,y) \in \bR^q \times \bR^{n-q}$ and $z=(z_1,z_2) \in \bR^q \times \bR^{n-q}$.
We start by proving that (i) implies (ii) in both settings, which is the hardest part.
So let $\mu = \alpha \delta_0^n + f \lambda^q \otimes \zeta_C^{n-q}$ be invertible.
The characteristic function of $\mu$ is then given by
\begin{eqnarray}
\widehat{\mu}(z_1,z_2) & = & \alpha + \int_{\bR^n} \re^{\ri z_1^T  x} \, \re^{\ri z_2^T y} \, f(x,y) \, (\lambda^q \otimes  \zeta^{n-q}_C)(\di (x,y)) \nonumber \\
& = & \alpha\, \re^{\ri z_2^T 0_{n-q}} + \sum_{y\in C} \re^{\ri z_2^T y} \int_{\bR^q}
\re^{\ri z_1^T x} f(x,y) \, \lambda^q (\di x), \quad (z_1,z_2) \in \bR^q \times \bR^{n-q},
\label{eq-mu-zwei}
\end{eqnarray}
where the sum converges absolutely.
Denote by $A$ the set of all characteristic functions of elements in $\bC \delta_0^q + L(\bR^q)\subset M(\bR^q)$ (cf. Example~\ref{ex-Banach-algebra-1}~(c)), i.e. $$A := \left\{ \left( \bR^q \to \bC, z_1 \mapsto \int_{\bR^q} \re^{\ri z_1^T x} \, \rho(\di x) \right) : \rho \in \bC \delta_0^q + L(\bR^q) \right\}.$$
Then  $A$ is a unital commutative Banach algebra with respect to multiplication and the constant function 1 as identity $e_A$, which is isometrically isomorphic to $\bC \delta_0^q + L(\bR^q)$, with the norm $\|\cdot\|_A$ induced from the total variation norm on $\bC \delta_0^q + L(\bR^q)$. Since
$\bC \delta_0^q + L(\bR^q)$
is semisimple with connected Gelfand space by Example~\ref{ex-Banach-algebra-2}~(c), the same is true for $A$.
Equation \eqref{eq-mu-zwei} now implies  that there is a sequence $(a_y)_{y\in C \cup \{0\}}$ in $A$ with $\sum_{y\in C \cup \{0\}} \|a_y\|_A < \infty$ such that
\begin{equation} \label{eq-mu-drei}
\widehat{\mu}(\cdot, z_2) = \sum_{y \in C \cup \{0\}} \re^{\ri z_2^T y} a_y.
\end{equation}
In other words, $z_2 \mapsto \widehat{\mu}(\cdot, z_2)$ is in $\GFSQ$.
We assumed that $\mu$ is invertible. Denote its inverse by $\mu'$. Then $\mu'$ is of the same form as $\mu$ by Proposition~\ref{p-Taylor4}. Hence, by the same argument, the function $z_2 \mapsto \widehat{\mu'}(\cdot, z_2)$ is in $\GFSQ$, too. Since $\widehat{\mu'}(z_1,z_2) \widehat{\mu}(z_1,z_2)  =1$ for all $(z_1,z_2) \in \bR^q \times \bR^{n-q}$, we conclude that $z_2 \mapsto \widehat{\mu}(\cdot,z_2)$ is invertible in $\GFSQ$.
By the equivalence of (i) and (vii) in Theorem~\ref{t-characterisation}, we conclude that there are some countable set $D\subset \bR^{n-q}$, an absolutely summable sequence $(b_y)_{y\in D}$ in $A$ and some $\gamma_2 \in \bR^{n-q}$ such that
$$\widehat{\mu}(\cdot, z_2) = \widehat{\mu}(\cdot, 0_{n-q}) \,
\exp \left\{ \ri \gamma_2^T z_2 \, e_A+  \sum_{y\in D} \re^{\ri z_2^T y}  b_y \right\} \quad \forall\; z_2 \in \bR^{n-q}.$$
Since $b_y\in A$, there are $\beta_y \in \bC$ and $g_y \in L^1(\bR^q)$ such that $b_y = (\beta_y \delta_0^q + g_y \lambda^q)^\wedge$. The total variation of $\beta_y \delta_0^q + g_y \lambda^q$ is $|\beta_y| + \int_{\bR^q} |g_y(x)| \, \lambda^q (\di x)$, and since $(b_y)$ is absolutely summable, we conclude that $\sum_{y\in D} |\beta_y| < \infty$ and $\sum_{y\in D} |g_y| \in L^1(\bR^q)$.  Defining the (finite) complex measure
$$Q :=  \sum_{y\in D} (\beta_y \delta_0^q  + g_y \lambda^q) \otimes \delta_y^{n-q}$$
on $(\bR^n,\cB(\bR^n))$, we conclude with $\gamma := (0_q,\gamma_2)^T \in \bR^n$ that
\begin{align*}
\widehat{\mu}(z_1,z_2) &
= \widehat{\mu}(z_1,0_{n-q}) \exp \left\{ \ri \gamma^T \left( \begin{array}{c} z_1 \\ z_2 \end{array} \right) + \int_{\bR^n} \re^{\ri (z_1^T x + z_2^T y)} Q(\di (x,y)) \right\} \\
& =  \widehat{\mu}(z_1,0_{n-q})\, \re^{Q(\bR^n)} \exp \left\{ \ri \gamma^T \left( \begin{array}{c} z_1 \\ z_2 \end{array} \right) + \int_{\bR^n} \left( \re^{\ri (z_1^T x + z_2^T y)}-1\right) Q(\di (x,y)) \right\}
\end{align*}
for all $(z_1,z_2) \in \bR^q \times \bR^{n-q}$. By \eqref{eq-mu-zwei},  $\widehat{\mu}(\cdot, 0_{n-q})$ is the characteristic function of $\alpha \delta_0^q + (\sum_{y\in C} f(\cdot, y)) \lambda^q \in M(\bR^q)$, and this complex measure is invertible (since $\mu$ is invertible with inverse $\mu'$, we have $\widehat{\mu}(z_1,z_2) \widehat{\mu'}(z_1,z_2) = 1$, hence $\widehat{\mu}(\cdot ,0_{n-q}) \widehat{\mu'}(\cdot,0_{n-q}) = e_A$, so that the inverse has characteristic function $\widehat{\mu'}(\cdot,0_{n-q})$). From Propositions~\ref{p1-q=n}~(ii),~\ref{p2-q=n}~(v) and Theorem~\ref{t-delta-lebesgue1}~(v') we conclude
 that there are some $c\in \bC \setminus \{0\}$, $g\in L^1(\bR^q)$ and (when $q=1$) some $m\in \bZ$ such that
$$\widehat{\mu}(z_1,0_{n-q}) = c\, \exp \left\{ (\re^{\ri z_1^T x} - 1) \left(g(x) \lambda^q(\di x) + \one_{\{q=1\} }\,  m \frac{\re^{-|x|}}{x} \lambda^1(\di x) \right) \right\}$$
for all $z_1\in \bR^q$. Denoting
$$Q' := \left(g(x) \lambda^q(\di x) + \one_{\{q=1\} }\,  m \frac{\re^{-|x|}}{x} \lambda^1(\di x) \right) \otimes \delta_0^{n-q},$$
we conclude that
$$\widehat{\mu}(z_1,z_2) = c\, \re^{Q(\bR^n)}  \exp \left\{
\ri \gamma^T \left( \begin{array}{c} z_1 \\ z_2 \end{array} \right)
 + \int_{\bR^n} ( \re^{\ri z_1^T x + \ri z_2^T y} -1) (Q+Q')(\di (x,y)) \right\} $$
for all $(z_1,z_2) \in \bR^q \times \bR^{n-q}$,
showing that $\mu \in \CLK$ with $\gamma = (0_q,\gamma_2)^T$ and complex L\'evy type measure $\nu:= Q+Q'$, which has the required form. Thus we have proved that (i) implies (ii) in both cases (a) and (b).

To see that (ii) implies (iii), in case (b) define additionally $\nu_0 := \nu$ and $\nu_1 := 0_{M(\bR^n)}$, and let $\widehat{\mu}(z) = c \exp \left\{ \ri \gamma^T z + \int_{\bR^n} \left( \re^{\ri z_1^T x  + \ri z_2^T y} - 1 \right) (\nu_0 + \nu_1)(\di (x,y))\right\}$ with $c\neq 0$ in both cases (again we write $z=(z_1,z_2)$).  From the discussion following Equation~\eqref{eq-exponential3}, we see that $$\bR^n \to \bC, \quad z\mapsto c \, \exp \left\{ \ri \gamma^T z + \int_{\bR^n} \left( \re^{\ri z_1^T x + \ri z_2^T y} - 1 \right) \nu_0 (\di (x,y))\right\}$$ is the characteristic function of $\delta_\gamma^n \ast \exp (\nu')$ for some suitable $\nu'\in M(\bR^n)$,
thus already proving (iii) in case (b). In case (a), when $q=1$, a simple calculation (which for the second equality is also given in the proof of Theorem~4.4 in \cite{B19}) gives
\begin{align}
& \exp \left( \int_{\bR} \int_{\bR^{n-1}} \left( \re^{\ri (x^T z_1 + y^T z_2)} -1\right) m \frac{\re^{-|x|}}{x} \, \delta_0^{n-1}(\di y)\, \lambda^1(\di x) \right) \label{eq-sigma-cf} \\
= &  \exp \left( \int_\bR \left( \re^{\ri x z_1} -1\right) m \frac{\re^{-|x|}}{x} \, \lambda^1(\di x) \right) \nonumber\\
= &  \left( \frac{z_1-\ri}{z_1+\ri} \right)^m (-1)^m = \left( \frac{1+\ri z_1}{1-\ri z_1}\right)^m = (\widehat{\sigma}(z_1))^m = (\sigma^{\ast m} \otimes \delta_0^{n-1})^\wedge  (z_1,z_2) .\nonumber
\end{align}
Hence (ii) implies (iii) also in case (a).

Finally, to show that (iii) implies (i), it is enough to notice that $\delta_\gamma^n$, $\exp (\nu')$ and $\sigma^{\ast m} \otimes \delta_0^{n-1}$ are all invertible (with inverses $\delta_{-\gamma}^n$, $\exp(-\nu')$ and $\sigma^{\ast (-m)}\otimes \delta_0^{n-1}$, respectively), hence so is their convolution.
\end{proof}

\begin{remark} \label{no-Cramer-Wold}
While in the cases $q=n$ and $q=0$ we obtained Cram\'er--Wold devices for invertibility and for membership in $\mathrm{CLK}_0$ in Corollary~\ref{c-Cramer-Wold} and Theorem~\ref{t-Khartov-Cramer-Wold}, we do not know whether similar results hold when $q\in \{1,\ldots, n-1\}$. The proofs for $q=n$ and $q=0$ where based on the equivalent property (ii) in Proposition~\ref{p1-q=n} and Theorem~\ref{t-Khartov}, respectively, which is not available when $q\in \{1,\ldots, n-1\}$. Similarly, the question is unsettled for general complex measures.
\end{remark}

\section{Proof of the main results} \label{S8}

We can now give the proofs of our main results stated in Section~\ref{S2}.

\begin{proof}[Proof of Theorem~\ref{t-main1}]
Suppose that $\mu \in M(\bR^n)$ has representation~\eqref{eq-repr1} with $\gamma$, $p$, $U_i$, $m_i$ and $\nu$ as given there. Observe that $\sigma\in M(\bR)$ is invertible with inverse $\sigma^{-1}$. Define
$$\mu' := \delta_{-\gamma}^n \ast U_1  (\sigma^{\ast (-m_1)} \otimes \delta_0^{n-1}) \ast \ldots \ast
    U_p  (\sigma^{\ast (-m_p)} \otimes \delta_0^{n-1}) \ast \exp(-\nu).$$
Then $\mu' \in M(\bR^n)$ and it is easily checked that
$$U_i (\sigma^{\ast (-m_i)} \otimes \delta_0^{n-1}) \ast  U_i (\sigma^{\ast (m_i)} \otimes \delta_0^{n-1}) = U_i (\delta_0^n) = \delta_0^n,$$
so that $\mu$ is invertible with inverse $\mu'$.

Conversely, suppose that $\mu\in M(\bR^n)$ is invertible with inverse $\mu'$. By Proposition~\ref{p-Taylor3}, there exist $p$, $\nu$, $U_1,\ldots, U_p$, $q_1,\ldots,q_p$, $C_1,\ldots, C_p$, $\alpha_1,\ldots, \alpha_p$, $f_1,\ldots, f_p$, $\mu_1,\ldots, \mu_p$ as specified there such that
$$\mu = \mu_1\ast \ldots \ast \mu_p \ast \exp(\nu).$$
The $\mu_i$ are invertible and of the form $\mu_i = U_i (\alpha_i \delta_0^n + f_i \, \lambda^{q_i} \otimes \zeta^{n-q_i}_{C_i})$. Since $U_i$ defines a bijective linear mapping, each of the
$\alpha_i \delta_0^n + f_i \, \lambda^{q_i} \otimes \zeta^{n-q_i}_{C_i} = U_i^{-1} (\mu_i)$ is invertible as well. Now if $q_i=0$ or $q_i \geq 2$, then by Theorems~\ref{t-Khartov}~(vi) and~\ref{t-allg-q}~(b)~(iii), there exist $\gamma_i \in \bR^n$ and $\nu_i' \in M(\bR^n)$ such that  $U_i^{-1}(\mu_i ) = \delta_{\gamma_i}^n \ast \exp(\nu_i')$, and it is easily checked that then
$$\mu_i = U_i (\delta_{\gamma_i}^n) \ast U_i(\exp (\nu_i')) = \delta_{U_i(\gamma_i)}^n \ast \exp (U_i (\nu_i')).$$
If $q_i = 1$, then by Proposition~\ref{p2-q=n}~(vi) and Theorem~\ref{t-allg-q}~(a)~(iii) there exist $\gamma_i \in \bR^n$, $m_i \in \bZ$ and $\nu_i' \in M(\bR^n)$ such that $U_i^{-1}(\mu_i ) = \delta_{\gamma_i}^n \ast (\sigma^{\ast m_i} \otimes \delta_0^{n-1}) \ast \exp(\nu_i')$, so that
$$\mu_i = U_i (\delta_{\gamma_i}^n) \ast U_i (\sigma^{\ast m_i} \otimes \delta_0^{n-1}) \ast U_i(\exp (\nu_i')) = \delta_{U_i(\gamma_i)}^n \ast  U_i (\sigma^{\ast m_i} \otimes \delta_0^{n-1}) \ast \exp (U_i (\nu_i')).$$ This shows that $\mu$ has representation~\eqref{eq-repr1} with $\gamma := U_1(\gamma_1) + \ldots + U_p(\gamma_p)$ and $\nu$ replaced by $\nu + U_1(\nu_1') + \ldots + U_p(\nu_p')$, thus finishing the proof.
\end{proof}

\begin{proof}[Proof of Corollary~\ref{c-main2}]
The \lq if\rq~part is clear from Theorem~\ref{t-main1}, and it is easily checked that the inverse of a finite signed measure must be signed as well. For the \lq only if\rq~part, let $\mu\in M(\bR^n)$ be signed and invertible. Replacing $\mu$ by $(\mu(\bR^n))^{-1} \mu$, we can assume that $\mu(\bR^n) = 1$.  By Theorem~\ref{t-main1}, $\mu$ has representation \eqref{eq-repr1} with the specified quantities, in particular with $\nu\in M(\bR^n)$ (possibly complex-valued, and it has to be shown that $\nu$ can be chosen to be real-valued). Then
$$\exp(\nu) = \delta_{-\gamma}^n \ast U_1  (\sigma^{\ast (-m_1)} \otimes \delta_0^{n-1}) \ast \ldots \ast
    U_p  (\sigma^{\ast (-m_p)} \otimes \delta_0^{n-1}) \ast \mu$$
is a finite signed measure with $(\exp(\nu)) (\bR^n) = \mu(\bR^n) = 1$ (observe that $\sigma(\bR) = 1$). The characteristic function of $\exp(\nu)$ is given in \eqref{eq-exponential2}, so that $\exp(\nu)\in \CLK$ with complex L\'evy type measure given by $\nu$ restricted to $\cB^n_0$. By Remark~\ref{r-qid-uniqueness}~(c), $\nu$ restricted to $\cB^n_0$ is real-valued. Hence $\nu' := \nu - \nu(\{0\}) \delta_0^n$ is real-valued. Since
$$1 = (\exp(\nu))(\bR^n) = \re^{\nu(\bR^n)} = \re^{\nu'(\bR^n)} \re^{\nu(\{0\})}$$
and $\nu'(\bR^n) \in \bR$ we conclude that the imaginary part of $\nu(\{0\})$ is an integer multiple of $2\pi$, say $2\pi k$. Since $\exp (\nu - 2\pi \ri k \delta_0^n) = \exp(\nu)$ we can replace $\nu$ in the representation \eqref{eq-repr1} by
$\nu - 2\pi \ri k \delta_0^n$, which is then a finite signed measure.
\end{proof}

\begin{proof}[Proof of Theorem~\ref{t-main3}]
(a,d) For each $u\in S^{n-1}$ denote by $\rho_u$ the finite signed measure on $S^{n-1}$ given by $\rho_u := \delta_u - \delta_{-u}$ and by $\tau_u$ the quasi-L\'evy type measure on $\bR^n$ given by
$$\tau_u(B) := \int_{S^{n-1}} \int_0^\infty \one_B(r\xi) r^{-1} \re^{-r} \, \di r \, \rho_u(\di \xi), \quad B \in \cB_0^n.$$
Denoting by $e_1=(1,0,\ldots,0)^T \in \bR^n$ the first unit vector in $\bR^n$, we easily see that
$\tau_{e_1}$ is the product measure of $\re^{-|x|}/x \, \lambda^1(\di x)$ on $\bR$ with $\delta_0^{n-1}(\di y)$ on $\bR^{n-1}$. Equation \eqref{eq-sigma-cf} then shows that the characteristic function of $\sigma^{\ast m} \otimes \delta_0^{n-1}$ at $z\in \bR^n$ is given by $\exp(\int_{\bR^n} (\re^{\ri w^T z} -1) m\tau_{e_1} (\di w))$. Now if $U$ is an orthogonal matrix, then
\begin{align}
\left[ U (\sigma^{\ast m} \otimes \delta_0^{n-1}) \right]^{\wedge} (z) & =
\left[ \sigma^{\ast m} \otimes \delta_0^{n-1} \right]^{\wedge} (U^T z) \label{eq-cf-U} \\
& = \exp \left( \int_{\bR^n} ( \re^{\ri w^T U^T z} - 1) m \tau_{e_1} (\di w) \right) \nonumber \\
& =  \exp \int_{\bR^n} (\re^{\ri v^T z} -1) m \tau_{U(e_1)}(\di v) \quad \forall\; z \in \bR^n, \nonumber
\end{align}
the latter since the image measure $U(\tau_{e_1})$ of $\tau_{e_1}$ under the mapping $U$ is easily seen to be equal to $\tau_{U(e_1)}$.

Now let $\mu\in M(\bR^n)$ be invertible. By Theorem~\ref{t-main1} there are $\nu \in M(\bR^n)$, $\gamma\in \bR^n$, $p\in \bN_0$, $m_1,\ldots, m_p \in \bZ$ and $U_1,\ldots, U_n \in O(n)$ such that $\mu$ has representation~\eqref{eq-repr1}. By \eqref{eq-cf-U} and \eqref{eq-exponential2}, the characteristic function of $\mu$ is then given by
$$\widehat{\mu}(z) = \re^{\nu(\bR^n)} \exp \left\{ \ri \gamma^T z + \int_{\bR^n} \left( \re^{\ri z^T x} -1 \right) (\nu + m_1 \tau_{U_1(e_1)} + \ldots + m_p \tau_{U_p(e_1)}) (\di x)\right\},$$
so that $\widehat{\mu}$ has representation \eqref{eq-main2} with $\nu_0 := \nu$ and $\nu_1 := \sum_{j=1}^p m_j \tau_{U_j(e_1)}$. Setting $u_j := U_j(e_1)$, we see that $\sum_{j=1}^p m_j \rho_{u_j}$ agrees with $\Lambda$ as defined in~\eqref{eq-main4}, so that $\nu_1$ has indeed representation \eqref{eq-main3}. This proves the \lq only if\rq-parts in (a) and (b).

For the \lq if\rq-parts, let $f:\bR^n\to \bC$ be defined by the right-hand side of \eqref{eq-main2}, with $c,p, \gamma,\nu_0,\nu_1,\Lambda$, $m_1,\ldots, m_p$ and $u_1,\ldots, u_p$ as specified in Theorem~\ref{t-main3}. Choose orthogonal matrices $U_1,\ldots, U_p$ such that $U_j(e_1) = u_j$ and define $\widetilde{\mu} \in M(\bR^n)$ by
$\widetilde{\mu} = \delta_{\gamma}^n \ast U_1  (\sigma^{\ast m_1} \otimes \delta_0^{n-1}) \ast \ldots \ast
    U_p  (\sigma^{\ast m_p} \otimes \delta_0^{n-1}) \ast \exp(\nu_0)$. Then $\widetilde{\mu}$ is invertible by Theorem~\ref{t-main1}. As seen above, its characteristic function is given by
$$\widehat{\widetilde{\mu}}(z) = \re^{\nu_0(\bR^n)} \exp \left\{ \ri \gamma^T z + \int_{\bR^n} \left( \re^{\ri z^T x } - 1\right) (\nu_0 + \nu_1)(\di x) \right\} = \frac{\re^{\nu_0(\bR^n)}}{c} f(z).$$
This shows that $f$ is the characteristic function of the invertible complex measure $c\,\re^{-\nu_0(\bR^n)} \widetilde{\mu}$, finishing the proofs of (a) and (d).

(b) That $c$ in \eqref{eq-main2} is equal to $\mu(\bR^n)$ was already observed in Remark~\ref{r-qid-uniqueness}~(b). Since $x \mapsto \re^{\ri z^T x} - 1$ takes the value 0 for $x=0$ it follows that in \eqref{eq-main2} we can replace $\nu_0$ by $\nu_0 - \nu_0(\{0\}) \delta_0^n$.

(c) If $\mu$ is an invertible signed measure, then by (b) we can choose $\nu_0$ to satisfy $\nu_0(\{0\}) = 0$, and then it follows from Remark~\ref{r-qid-uniqueness}~(c) that the complex L\'evy type measure $\nu_0+\nu_1$  (restricted to $\cB_0^n)$ is real-valued. Since $\nu_1$ is real-valued, so is $\nu_0$. That $c=\mu(\bR^n) \in \bR$ is clear.
\end{proof}

\section{Divisibility properties of invertible complex mesures} \label{S9}

Recall that a probability measure $\mu$ on $\bR^n$ is infinitely divisible
if  for every $k\in \bN$ there exists a \emph{probability measure} $\mu_{1/k}$ on $\bR^n$ such that $\mu_{1/k}^{\ast k} = \mu$, i.e.~if $\mu$ has convolution roots of all orders, where the convolution roots are probability measures. Infinitely divisible probability measures are well known to be characterised by the L\'evy--Khinitchine formula, see e.g.~\cite[Thm.~8.1]{Sato2013} or \cite[Thm.~9.13]{BrockLind2024}. It is now natural to look for probability measures $\mu$, or more generally for complex measures $\mu\in M(\bR^n)$, for which for all $k\in \bN$ there exists a \emph{complex measure $\mu_{1/k} \in M(\bR^n)$} such that $\mu_{1/k}^{\ast k} = \mu$. This can be seen as an infinite divisibility property \emph{within the class $M(\bR^n)$.} Inspired by Sz\'ekely \cite{Szekely}, considerations like this have been carried out in Kerns~\cite{Kerns}, who restricts to finite signed rather than complex measures  on $\bR$ and calls the corresponding signed measures \emph{generalized infinitely divisible}, see \cite[Def. 2.1.1]{Kerns}. He  obtains several examples of such distributions. Similar considerations for discrete signed measures on $\bN_0$ have been carried out in Zhang et al.~\cite{ZhangLiKerns}, who give a characterisation of such signed measures, see~\cite[Thm.~2.1]{ZhangLiKerns}.
We shall now have a look at complex measures on $\bR^n$ which have complex convolution roots of all orders and are additionally invertible.

\begin{theorem}[Invertible complex measures having convolution roots of all orders] \label{t-divisibility}
Let $\mu\in M(\bR^n)$. Then the following are equivalent:
\begin{enumerate}
\item[(i)] $\mu$ is invertible and has convolution roots in $M(\bR^n)$ of all orders, i.e.~for every $k\in \bN$ there exists some $\mu_{1/k} \in M(\bR^n)$ such that $\mu_{1/k}^{\ast k} = \mu$.
\item[(ii)] $\widehat{\mu}$ is zero-free and $z\mapsto (\mu(\bR^n))^t \exp(t\psi(z))$ defines for every $t\in \bR$ the characteristic function of some complex measure $\mu_t$, where $\psi$ denotes the distinguished logarithm of $(\mu(\bR^n))^{-1} \widehat{\mu}$ and $(\mu(\bR^n))^t$ is some complex $t^{\rm th}$-power of $\mu(\bR^n)$.
\item[(iii)] There exist $\gamma\in \bR^n$ and $\nu\in M(\bR^n)$ such that $\mu = \delta_\gamma^n \ast \exp(\nu)$.
\item[(iv)] $\mu \in \CLK$ with finite complex L\'evy type measure.
\end{enumerate}
\end{theorem}

If $\widehat{\mu}$ is zero-free and $\psi$ is the distinguished logarithm of $\widehat{\mu}/(\mu(\bR^n))$, then $z\mapsto \exp (t \psi(z))$ is the \emph{distinguished $t^{\rm th}$-power of $\widehat{\mu}/(\mu(\bR^n))$} for each $t\in \bR$, see \cite[Def.~10.24]{BrockLind2024}. If this is  the characteristic function of some complex measure, then so is $z\mapsto (\mu(\bR^n))^t \exp(t\psi(z))$, which corresponds to a complex measure, $\mu_t$ say. Characterisation (ii)  in Theorem~\ref{t-divisibility} hence means that $\mu_{t}\in M(\bR^n)$ is definable for all $t\in\bR$. Observe however that $\mu_t$ is not unique due to the (in general) not unique factor $(\mu(\bR^n))^t$.

\begin{proof}
To show that (i) implies (iv), let $\mu$ be invertible and have complex convolution roots $\mu_{1/k}$ of all orders $k\in \bN$. By Theorem~\ref{t-main3}, $\widehat{\mu}$ has representation \eqref{eq-main2} with the quantities $c, \gamma, \nu_0,\nu_1,\Lambda,m_j, u_j$ as specified there. Observe that we can without loss of generality choose $p$ and the $u_1,\ldots, u_p$ such that $\{u_1,-u_1\}, \ldots, \{u_p,-u_p\}$ are pairwise disjoint. Since $(\widehat{\mu_{1/k}}(z))^k = \widehat{\mu}(z)$, we have
$$\widehat{\mu_{1/k}}(z) = c^{1/k} \,  \exp\left \{ \ri k^{-1} \gamma^T z + \int_{\bR^n} \left( \re^{\ri z^T x} -1\right) \, k^{-1} ( \nu_0 + \nu_1) (\di x)\right\}$$
for all $z\in \bR^n$ and $k\in \bN$, where $c^{1/k}$ is some complex $k^{\rm th}$-root of $c$. Observe that $k^{-1} \nu_0 \in M(\bR^n)$ is finite and $k^{-1} \nu_1$ has representation \eqref{eq-main3} with $\Lambda$ replaced by $k^{-1} \Lambda = \sum_{j=1}^p \left( \frac{m_j}{k} \delta_{u_j}^{S^{n-1}} - \frac{m_j}{k} \delta_{-u_j}^{S^{n-1}} \right)$. Now if $\Lambda$ is not the zero-measure and $k> \max\{|m_1|,\ldots,|m_p|\}$, then the $m_j/k$ will no longer be integers, and it is easily seen that $k^{-1} \nu_0 + k^{-1} \nu_1$ cannot be written as $\nu_0' + \nu_1'$ for some $\nu_0'\in M(\bR^n)$ and $\nu_1'$ of the form \eqref{eq-main3} with $\Lambda$ replaced by $\Lambda'$ having the form $\Lambda' = \sum_{j=1}^{p'} \left( m_j' \delta_{u_j'}^{S^{n-1}} - m_j' \delta_{-u_j'}^{S^{n-1}} \right)$ with integers $m_1', \ldots, m_{p'}'$. Hence, if $\Lambda$ is not the zero-measure, then $\mu_{1/k}$ will not be invertible for large $k$ by Theorem~\ref{t-main3}. However, since
$$\mu_{1/k} \ast (\mu_{1/k}^{\ast (k-1)} \ast \mu') =
\mu\ast \mu' = \delta_0^n,$$
where $\mu'$ denotes the inverse of $\mu$, we conclude that $\mu_{1/k}$ is invertible. Consequently,~$\Lambda$ must be the zero-measure, so that $\widehat{\mu}$ has representation \eqref{eq-main2} with $\nu_1=0$, thus proving~(iv).

That (iv) implies (iii) was already observed after equation~\eqref{eq-exponential3}. To see that (iii) implies (ii), observe that $\widehat{\mu}$ is clearly zero-free. For each $t\in \bR$ define $\mu_t := \delta_{t\gamma}^n \ast \exp(t\nu)$. Then $\mu_t \in M(\bR^n)$ and by \eqref{eq-exponential2} we have
$$\widehat{\mu_t}(z) = \re^{t\nu(\bR^n)} \exp \left\{ t \left( \ri \gamma^T z + \int_{\bR^n} \left( \re^{\ri z^T x} - 1 \right) \, \nu(\di x) \right) \right\} = \re^{t\nu(\bR^n)} \exp \left\{ t \psi(z) \right\}.$$ Hence $z\mapsto (\mu(\bR^n))^t \exp(t\psi(z))$ is a characteristic function of some complex measure.

Finally, to see that (ii) implies (i), for $k\in \bN$ define $\mu_{1/k}$ to be the complex measure having characteristic function $z\mapsto \widehat{\mu_{1/k}}(z) = (\mu(\bR^n))^{1/k} \exp (k^{-1} \psi(z))$. Then $(\widehat{\mu_{1/k}}(z))^k = \widehat{\mu}(z)$, so that $\mu_{1/k}$ is a $k^{\rm th}$ convolution root of $\mu$. Define $\mu_{-1}$ to be the complex measure having characteristic function $z\mapsto \widehat{\mu_{-1}}(z) = (\mu(\bR^n))^{-1} \exp (- \psi(z))$. Then $\widehat{\mu_{-1}}(z) \widehat{\mu}(z) = 1$, so that $\mu_{-1} \ast \mu = \delta_0^n$, showing that $\mu$ is invertible.
\end{proof}

\begin{example}\label{ex-divisibility}
Suppose that $\mu\in M(\bR^n)$ is discrete with $\inf_{z\in \bR^n} |\widehat{\mu}(z)| > 0$. Then $\mu\in \CLK$ with finite complex L\'evy type measure by Theorem~\ref{t-Khartov}, so that $\mu$ has complex convolution roots of all orders and $\mu^{\ast t} \in M(\bR^n)$ is definable for all $t\in \bR$ by Theorem~\ref{t-divisibility}. Similarly, when $\mu \in \bC \delta_0^n + L(\bR^n)$ with $\inf_{z\in \bR^n} |\widehat{\mu}(z)| > 0$, then $\mu$ is invertible by Proposition~\ref{p1-q=n}, but $\mu^{\ast t} \in M(\bR^n)$ can be defined for all $t\in \bR$ only when $n\geq 2$ (cf. Theorem~\ref{t-delta-lebesgue1}), or $n=1$ and $m=0$ in Proposition~\ref{p2-q=n}. Finally, if $\mu$ is a probability measure on $\bR^n$ having an atom of mass greater than $1/2$, then $\mu\in \CLK$ with finite quasi-L\'evy type measure by Cuppens~\cite[Thm.~4.3.7]{Cuppens1975}, see also \cite[Thm.~3.2]{BKL21}, hence $\mu$ has complex convolution roots of all orders and $\mu^{\ast t}$ can be defined for all $t\in \bR$.
\end{example}

\end{document}